\documentclass[12pt]{amsart}

\usepackage[margin=2.7cm]{geometry}
\usepackage{multirow}
\usepackage[parfill]{parskip}
\usepackage{graphicx}
\usepackage{float}
\usepackage{amsmath}
\usepackage{amssymb}
\usepackage{amsthm}
\usepackage{array}
\usepackage{mathrsfs}
\usepackage[all]{xy}
\usepackage{fancyhdr}
\usepackage{enumitem}
\usepackage{subfig}
\usepackage{tabularx}
\usepackage{xcolor}
\usepackage{textcomp}

\theoremstyle{plain}
\newtheorem{theorem}{Theorem}[section]
\newtheorem*{theorem1}{Theorem}

\newtheorem{lemma}[theorem]{Lemma}
\newtheorem{prop}[theorem]{Proposition}
\newtheorem{cor}[theorem]{Corollary}

\theoremstyle{definition}

\newtheorem{definition}[theorem]{Definition}

\theoremstyle{remark}
\newtheorem{remark}[theorem]{Remark}

\DeclareMathOperator{\C}{\mathbb{C}}
\DeclareMathOperator{\R}{\mathbb{R}}
\DeclareMathOperator{\Lc}{\mathcal{L}}

\DeclareMathOperator{\stab}{Stab}
\DeclareMathOperator{\unstab}{Unstab}
\DeclareMathOperator{\spn}{span}
\DeclareMathOperator{\cK}{\mathcal{K}}

\DeclareMathOperator{\cH}{\mathcal{H}}
\DeclareMathOperator{\cF}{\mathcal{F}}
\DeclareMathOperator{\cG}{\mathcal{G}}
\DeclareMathOperator{\cJ}{\mathcal{J}}

\newcommand\lf{\mbox{\textleaf}}

\title{Arboreal Singularities in Weinstein Skeleta}
\author{Laura Starkston}

\begin{document}
	\begin{abstract}
		We study the singularities of the isotropic skeleton of a Weinstein manifold in relation to Nadler's program of arboreal singularities. By deforming the skeleton via homotopies of the Weinstein structure, we produce a Morse-Bott* representative of the Weinstein homotopy class whose stratified skeleton determines its symplectic neighborhood. We then study the singularities of the skeleta in this class and show that after a certain type of generic perturbation either (1) these singularities fall into the class of (signed Lagrangian versions of) Nadler's arboreal singularities which are combinatorially classified into finitely many types in a given dimension or (2) there are singularities of tangency in associated front projections. We then turn to the singularities of tangency to try to reduce them also to collections of arboreal singularities. We give a general localization procedure to isolate the Liouville flow to a neighborhood of these non-arboreal singularities, and then show how to replace the simplest singularities of tangency (those of type $\Sigma^{1,0}$) by arboreal singularities.
	\end{abstract}
	\maketitle
	
	\section{Introduction}
	
	Weinstein manifolds are open symplectic manifolds compatible with Morse theory, and are deformation equivalent to Stein manifolds \cite{E}. They were introduced by Eliashberg and Gromov \cite{EG} as a special class of convex symplectic manifolds compatible with the symplectic handle construction of Weinstein \cite{W}. A Weinstein manifold contains a core isotropic \emph{skeleton} consisting of the set of points which do not escape to infinity under the Liouville flow--equivalently (in the Weinstein case), the stable manifolds of all the zeros of the vector field. Any arbitrarily small neighborhood of the skeleton completely recovers the Weinstein manifold, but the neighborhood cannot always be recovered from the skeleton itself near complicated singularities that typically develop. A longstanding hope has been to show that each Weinstein manifold can be represented by a Weinstein structure with a simple enough class of singularities so that the neighborhood can be recovered directly from the skeleton, and the singularities can be combinatorially classified by a finite list of types in each dimension. In this paper we make significant progress in this direction.
	
	The class of singularities we aim to aim for our skeleton to have, was proposed by Nadler \cite{N1}, inspired by mirror symmetry. Kontsevich proposed a method to calculate the Fukaya category of a Weinstein manifold in terms of the singular topology of its skeleton, provided the singularities fall into a certain class which in this paper are referred to as $A_n$ singularities \cite{K}. Nadler determined that an extended class of singularities was necessary and accordingly defined \emph{arboreal singularities} and calculated microlocal sheaf invariants for this class \cite{N1}, with the goal of combinatorially calculating certain versions of Fukaya categories. From this perspective, calculations are performed using local front projections of singular Legendrian submanifolds. The global invariants are obtained algebraically as sheaf theoretic invariants built from this local data. In this article, we take a more geometric global perspective on the singularities of a skeleton, and find that the arboreal singularities arise naturally in this context as well. From our perspective, the $A_n$ singularities occur naturally where (Lagrangian thickened) cores of handles of sequentially decreasing index meet (pictorially demonstrated on the left side of figure \ref{fig:A3Handles}). The larger class of arboreal singularities are necessary to deal with normal crossings involving multiple handles of the same index (see the right side of figure \ref{fig:A3Handles}). See sections \ref{s:genclass} and \ref{s:model} for a complete description of how to understand all arboreal singularities as naturally occuring in the skeleta of Weinstein manifolds.
	
	\begin{figure}
		\centering
		\includegraphics[scale=.5]{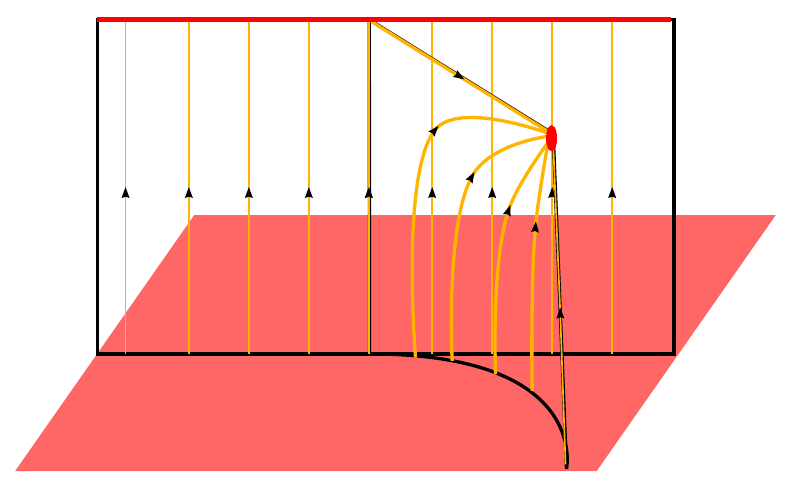} \hspace{1cm} \includegraphics[scale=.5]{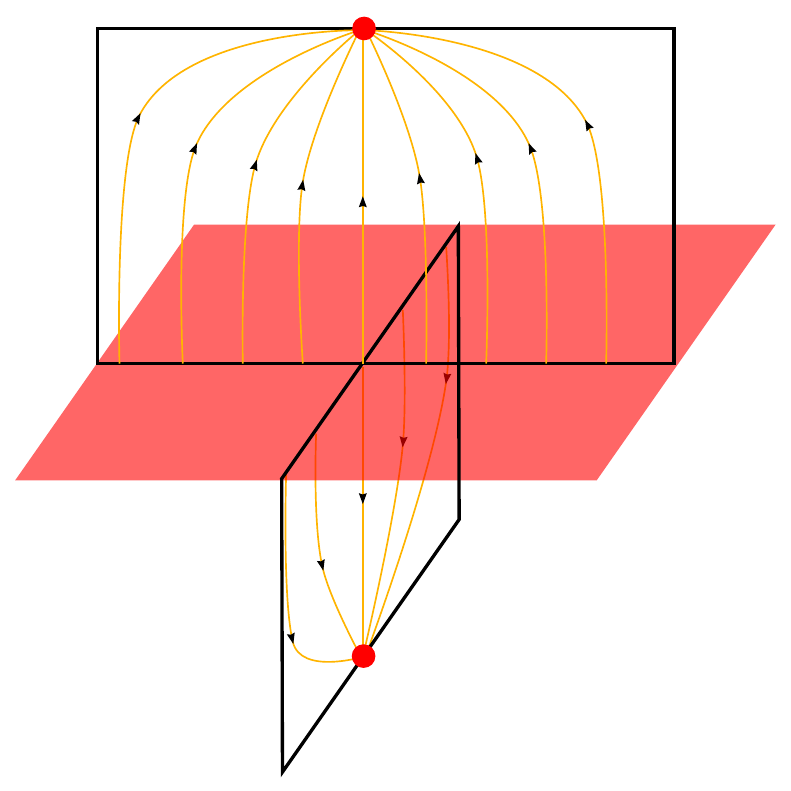}
		\caption{Left: an $A_3$ singularity as the union of the cores of an index $2$ handle and a Lagrangian thickened index $1$ and index $0$ handle. The red parts are critical loci. Right: an arboreal singularity occuring as the interaction of the cores of two index $2$ handles with a Lagrangian thickened index $0$.}
		\label{fig:A3Handles}
	\end{figure}
	
	One significant motivation and goal in the skeleton program is to enable us to define and calculate invariants of symplectic manifolds (which have often involved holomorphic curve counts with a range of foundational and computational challenges), by more topological methods. Sheaf-theoretic invariants have been proposed as alternative methods for computing Floer-theoretic invariants, and efforts continue to progress to establish equality between these two types of invariants in various versions (see work of Nadler and Zaslow \cite{NZ,N0} and Ganatra-Pardon-Shende \cite{GPS}). In this article, however, we stick with a more refined invariant of the Weinstein manifold: the skeleton up to Weinstein homotopy (note Weinstein homotopic implies symplectomorphic). In particular, any invariant of the Weinstein manifold defined in terms of the skeleton will remain unchanged (if it is in fact an invariant). Conversely, as this program develops further, we will have a finite set of moves connecting Weinstein homotopy arboreal skeleta and one could verify that a quantity \emph{is} an invariant by checking it is unchanged under these moves.
	
	In \cite{N2}, Nadler proved that any Legendrian singularity admits a non-characteristic deformation to a collection of arboreal Legendrian singularities preserving the associated microlocal sheaf category. Some have envisioned a version of this procedure for singular Lagrangians in a symplectic manifold. In this paper, while we maintain consistency with this general idea, we do not attempt to emulate this Legendrian arborealization in the Lagrangian setting, and it is possible that our procedure would yield different results than a rigorous global Lagrangian version of \cite{N2}. While there is not a unique arboreal representative of a given Weinstein class and thus could be more than one procedure for arborealization, our procedure has the advantage that the deformations of the Lagrangian skeleton we perform preserve the Weinstein homotopy type of the surrounding symplectic manifold instead of only (a priori) preserving its microlocal sheaf/wrapped Fukaya category invariants. Furthermore, our procedure is global and any new singularities appearing through this deformation at the global scale are carefully controlled.
	
	The main progress made in this paper comes in two stages given in sections \ref{s:bones} and \ref{s:tangential}. The goal of the first stage is to find a skeleton whose abstract structure uniquely determines the Weinstein manifold and is stable under certain kinds of perturbations (contact perturbations of the holonomy along contact type hypersurfaces). We define the key structures: bones and joints. Bones are smooth Lagrangian manifolds (possibly with boundary and non-compact ends) whose union is the entire skeleton. Each joint lies on a unique bone in the closure of other bones which approach it, and is given as a ``front projection'' of a Legendrian co-normal lift of the joint in a unit co-tangent bundle. Singularities in the skeleton are determined by the embeddings of the singular joints in the bones. 
	\begin{theorem}\label{thm:uniquenbhd}
		Every Weinstein manifold can be Weinstein homotoped so that its skeleton is built of bones meeting along joints. The diffeomorphism type of the (singular) joints embedded in the bones uniquely recovers the Weinstein manifold up to Weinstein homotopy.
	\end{theorem}
	
	There are two distinct causes of singularities in the joints: inductive accumulation of joints onto joints (think of handles of many different indices interacting), and failures of the joints to be immersed. In the first stage, we deal with the first issue.
	
	\begin{theorem}\label{thm:notanginf}[Technical statement is Theorem \ref{thm:notang}]
		If a skeleton of a Weinstein manifold is built of bones with immersed joints, then after a generic perturbation, the skeleton has only (signed, Lagrangian) arboreal singularities.
	\end{theorem}
	
	In the second stage, we turn towards singularities of tangency where a the front projection associated to a joint fails to be an immersion. We provide a general procedure (Proposition \ref{p:localize}) for localizing these singularities of tangency so that we can modify the skeleton in a neighborhood to try to replace the singularities of tangency with a collection of arboreal singularities without destroying the singularity structure of distant parts of the skeleton. Finally, we deal with that the simplest type of singularities of tangency ($\Sigma^{1,0}$ singularities using the notation of the Thom-Boardman stratification).
	
	\begin{theorem}\label{thm:sigma10}
		Suppose all of the tangential singularities of joints have type $\Sigma^{1,0}$. Then there is a homotopy of the Weinstein structure to one whose skeleton has only arboreal singularities.
	\end{theorem}
	
	A generic front projection in high dimension can have more complicated singularities than $\Sigma^{1,0}$ type. However, even dealing only with $\Sigma^{1,0}$ tangential singularities, we can cover a wide range of examples. In particular, this covers all $4$-dimensional ($n=2$) Weinstein manifolds. In higher dimensions, deeper singularities can occur generically in front projections, though sometimes global Legendrian isotopies can eliminate them -- see recent work of \'{A}lvarez-Gavela \cite{AG} which provides an h-principle that under a homotopical condition shows that a smooth Legendrian can be Legendrian isotoped so that its front projection has only $\Sigma^{1,0}$ singularities.

	The structure of this paper is as follows. Section \ref{s:background} gives the technical set-up and definitions, generalizations we require of definitions and lemmas from the literature, and the key Weinstein homotopy lemma. The reader may prefer to skip this section and refer back when necessary upon a first read through. The arborealization begins in section \ref{s:bones}, where we explain thickening and define bones and joints which provide the right language to discuss stable Lagrangian strata of the skeleton and their generic singularities. Then we prove in section \ref{s:genclass} that a skeleton whose joints have no singularities of tangency has only arboreal singularities (up to perturbation). In section \ref{s:tangential} we address singularities of tangency. In section \ref{s:localize}, we prove we can localize the Liouville flow to a neighborhood of the tangential singularities in a controlled manner so that modifications of the skeleton to eliminate singularities of tangency do not reverberate causing new non-arboreal singularities in parts of the skeleton which were already made arboreal. The second step is to modify this locally confined skeleton to eliminate the singularities of tangency. In section \ref{s:arborealize}, we provide this modification for the case of $\Sigma^{1,0}$ singularities of tangency yielding Theorem \ref{thm:sigma10}.

	\section*{Acknowledgments}
		This work has greatly benefited from many discussions with David Nadler and Yasha Eliashberg. I am grateful for David's invaluable intuition on arborealization which confirmed throughout when things were on track and corrected them when they were not. I have learned an enormous amount from Yasha and every discussion we have had has taught me a new way of thinking about Lagrangians, symplectic manifolds, and singularities. I have tried to incorporate some of these perspectives into my definitions and proofs, which I believe has significantly advanced the clarity and scope of these results. I am also grateful for advice, interest, shared knowledge, and suggestions from Daniel \'{A}lvarez-Gavela, Roger Casals, Kai Cieleibak, Josh Sabloff, Vivek Shende, and Alex Zorn. During the course of this work, I have been supported by an NSF Postdoctoral Fellowship Grant No. 1501728.
	
	\section{Technical set-up} \label{s:background} This section collects and adapts for the purposes of this article, definitions and lemmas on Weinstein manifolds, Weinstein homotopies, and arboreal singularities. The eager reader can skip to the main conceptual content starting section \ref{s:bones} and refer back for technical results and definitions in this section as needed. First we review basic definitions for Weinstein manifolds in section \ref{s:wein}. An essential generalization of the usual Morse Weinstein structures will allow Morse-Bott families with boundary, and we discuss this in detail in subsection \ref{s:MorseBott}. We discuss front projections as Lagrangian/Legendrian foliations in section \ref{s:front} and then review arboreal singularities and recast them in the Lagrangian setting in section \ref{s:arb}. The key to all of the Weinstein homotopies we will create is proved in section \ref{s:Whtpy}.
	
		\subsection{Liouville and Weinstein structures}\label{s:wein}
		We review here basic definitions of Weinstein manifolds. A more in depth discussion of Weinstein manifolds and their context in symplectic geometry can be found in \cite{EG}, \cite{CE}, and \cite{E}.
		
		A Liouville manifold is an exact symplectic manifold $(W,\omega)$, with a Liouville vector field $V$ which is $\omega$-dual to a primitive for $\omega$ ($\iota_V\omega=\eta$ and $d\eta=\omega$), such that $V$ is complete, and there is an exhaustion by compact domains $W=\cup_k W^k$ such that $V$ is outwardly transverse to the boundary of each $W^k$. 
		
		Given a Liouville manifold, let $\phi^t:W\to W$ denote the flow along $V$ for time $t$. Then, one can define its skeleton \cite{CE}: 
		$$Skel(W,\omega,V)=\cup_{k=1}^\infty \cap_{t>0} \phi^{-t}(W^k)$$
		which is independent of the exhaustion $\{W^k\}$.
		
		We will work with \emph{finite type} Liouville manifolds: those whose skeleta are compact.
		
		A Weinstein manifold is a Liouville manifold $(W,\omega, V)$ together with a generalized Morse function $\phi$ such that $V$ is gradient-like for $\phi$ (equivalently we say $\phi$ is a Lyapunov function for $V$). The strong version of this condition is defined to mean
		$$d\phi(V)\geq \delta(|V|^2+|d\phi|^2)$$
		for some $\delta: W\to \R_{>0}$, using some Riemannian metric on $W$ to define the norm. The weak version of the gradient-like/Lyapunov condition is that the zeros of $V$ coincide with the critical points of $\phi$, and off of this set $d\phi(V)>0$. The existence of a Lyapunov function for a given vector field is equivalent to the existence of a weak Lyapunov function in the complement of a neighborhood of the zero set of the vector field. Near the zeros, the existence of a Lyapunov function puts constraints on the behavior of the flow.
		
		At a zero $p$ of a vector field $V$, the linearization $D_pV$ splits the tangent space into invariant subspaces spanned by generalized eigenvectors corresponding to eigenvalues with positive, negative, or zero real part:
		$$T_pW = E_p^+\oplus E_p^- \oplus E_p^0.$$
		There are unique locally smooth $V$ invariant manifolds whose tangent spaces at $p$ are given by $E_p^\pm$, called the stable and unstable manifolds.

		For a Weinstein manifold where the function $\phi$ is Morse, the skeleton is the union of the stable manifolds of the zeros of the Liouville vector field.  Each such stable manifold is isotropic. Generalizing this situation, we define an \emph{isotropically stratified skeleton} of a Liouville manifold, to be the skeleton of a Liouville manifold together with a stratification such that each stratum is isotropically embedded in $(W,\omega)$. In the case of a Morse Weinstein structure, the dimension of each stratum (stable manifold) is the index of the corresponding critical point.
		
		We will work with a mild generalization of Weinstein structures which does not require the zeros of $V$ (equivalently critical points of $\phi$) to be isolated. We will require the zeros to be a closed subset of $W$, and the Liouville condition will ensure that the submanifold families of zeros will be isotropic (see section \ref{s:MorseBott}).
		
		The notion of equivalence we will work with is \emph{Weinstein homotopy}, meaning a 1-parameter family of Weinstein structures, but again Weinstein structure will refer to our mild generalization. In order to ensure that the topology and dynamics of the Liouville vector field are not sent off the infinite end, we require the Weinstein homotopy to be a composition of simple Weinstein homotopies $(\omega_t,V_t,\phi_t)$ which are each compatible with some compact exhaustion $\{W^k\}_{k=1}^{\infty}$ in the sense that $V_t$ is transverse and outward pointing along $\partial W_k$ for all $k$ and all $t$ parameters for the simple Weinstein homotopy.

		\subsection{Morse-Bott with boundary} \label{s:MorseBott}
		In order to spread out the singularities of a skeleton which can collect at subcritical points, we will utilize Weinstein structures with Morse-Bott families of critical points, where we allow the families to have boundary but control the behavior of $V$ near this boundary. This submanifold $Z$ with boundary where $V$ vanishes is necessarily isotropic by the Liouville condition. 
		
		For any vector field $V$, at a point $r$ where $V_r=0$, the differential of $V$ $d_rV$ splits $T_rW$ as $E_r^+\oplus E_r^-\oplus E_r^0$ where $E_r^{+/-/0}$ is spanned by the generalized eigenvectors corresponding to eigenvalues with $+/-/0$ real part. A zero of the vector field is called \emph{hyperbolic} if $E_r^0=\{0\}$ (note this is a slight generalization of a non-degenerate zero which rules out purely imaginary eigenvalues).
		
		We will work with a slightly more general notion than the non-degeneracy or hyperbolic condition on the zeros of a Liouville vector field in a ``Weinstein manifold.'' At points $p\in Z$ (including in $\partial Z$), since $V$ vanishes along every point in $Z$, $T_pZ\subset E_p^0$. For our slightly generalized notion of Weinstein, we will require that $T_pZ = E_p^0$ for every $p\in Z$ (including $p\in \partial Z$), to replace the non-degenerate condition. We still have smooth stable and unstable manifolds tangent to $E_r^-$ and $E_r^+$ respectively.
		
		For $q\in \partial Z$, we additionally require that in the direction in $T_qZ$ which is outwardly normal to the boundary, while the analytic germ at the point $q$ cannot detect it, the vector field is outward pointing, meaning there is a unique non-constant flow-line $\gamma(t)$ of $V$ such that $\lim_{t\to -\infty}\gamma(t)=q$, and such that the closure of the image of $\gamma$ is tangent at $q$ to $T_qZ$ in an outward normal direction to the boundary (see figure \ref{fig:bdryzeros}). We will call submanifolds with boundary where $V=0$ which have this property \emph{boundary repellent}.
		
		\begin{figure}
			\centering
			\includegraphics[scale=.5]{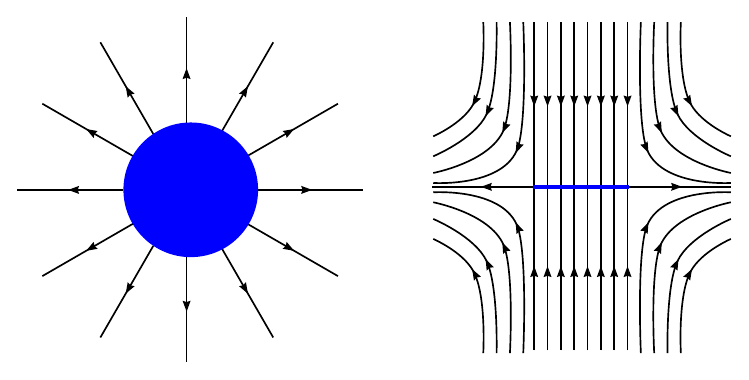}
			\caption{The behavior near submanifolds $Z$ with boundary where the Weinstein structure has zeros/critical points. Zeros are colored blue. Flow-lines are indicated in an isotropic slice containing $Z$. The flow is outward radial in undrawn transverse directions which pair symplectically with the coordinates in this plane.}
			\label{fig:bdryzeros}
		\end{figure}
		
		\begin{definition}
			A Weinstein structure $(V,\phi)$ on $(W,\omega)$ is Morse-Bott* if the zeros of $V$ come in connected components $Z_i$ which are submanifolds or submanifolds with boundary in $W$ such that for each $z\in Z$, $T_zZ=E_z^0$ and if $\partial Z_i\neq \emptyset$ then $Z_i$ is boundary repellent.
		\end{definition}
		
		We will always work in the class of Morse-Bott* Weinstein structures.
		
		Let $Z$ be a connected component of the subset of points where $V=0$. Let 
		$$\stab(Z)=\{w\in W\mid \lim_{t\to \infty}V^t(w)\in Z \}$$
		be the union of the stable manifolds of all the zeros in $Z$. 
		
		We verify that all stable manifolds in this Morse-Bott* setting are isotropic submanifolds (possibly with boundary).
		
		\begin{prop}\label{p:isotropic}
			Let $V$ be a Morse-Bott* Liouville vector field and $Z$ a connected component of zeros of $V$. Then the stable manifold $\stab(Z)$ is an isotropically embedded submanifold (possibly with boundary) locally near $Z$.
		\end{prop}
		
		\begin{proof}
			By the Morse-Bott* condition, $E_z^0=T_zZ$. In particular, the eigenvalues of $d_zV$ with vanishing real part are identically zero. Therefore, for $z\in Z$, the normal bundle splits as $\nu_+\oplus \nu_-$, where the fiber of $\nu_{\pm}$ is identified with $E^{\pm}_z$.
			
			The stable manifold $\Delta:=\stab(Z)$ is the image of an immersion of $\nu^-$ by Proposition 3.2 of \cite{AB}. There is a map $\pi: \Delta\to Z$ defined by $\pi(w)=\lim_{t\to\infty}V^t(w)$ which when restricted to a neighborhood of $Z$ is a fiber bundle with isotropic fibers, because the local stable manifold of any (possibly degenerate) zero of a Liouville vector field is isotropic \cite[Proposition 11.9]{CE}. We will extend this to ensure that the union of all these isotropic fibers is also an isotropic submanifold using local coordinates.
			
			$Z$ itself is also isotropic. Therefore, there is a symplectomorphism identifying a neighborhood of $Z$ with $T^*Z \times SN(Z)$ where $SN(Z)$ denotes the symplectic normal bundle of $Z$. For notational simplicity, we assume $SN(Z)$ is a trivial bundle, which will suffice for all our applications. (In general since we are proving a local statement, we can use a local trivialization.) Thus we get local coordinates $(\vec{q},\vec{p},\vec{x},\vec{y})\in T^*Z\times \C^k$ on a neighborhood of $Z$, with symplectic form $d\vec{q}\wedge d\vec{p}+d\vec{x}\wedge d\vec{y}$. The Liouville vector field $V$ vanishes along $Z=\{\vec{p}=\vec{x}=\vec{y}=0\}$. If in a coordinate chart of this neighborhood, 
			$$V=\sum F_i(\vec{q},\vec{p},\vec{x},\vec{y})\partial_{p_i}+ G_i(\vec{q},\vec{p},\vec{x},\vec{y})\partial_{q_i}+H_j(\vec{q},\vec{p},\vec{x},\vec{y})\partial_{x_j}+J_j(\vec{q},\vec{p},\vec{x},\vec{y})\partial_{y_j}$$
			then the Liouville condition $d(\iota_V\omega)=\omega$ requires
			$$\frac{\partial F_i}{\partial p_i}+\frac{\partial G_i}{\partial q_i}=1 \qquad \text{ for all }i.$$
			Since $V$ vanishes identically along $Z$, $F_i(\vec{q},0,0,0)=0$ and $\frac{\partial G_i}{\partial q_i}(\vec{q},0,0,0)=0$.
			It follows that $\frac{\partial F_i}{\partial p_i}(\vec{q},0,0,0)=1$ for all $i$, so in a sufficiently small neighborhood of $Z$, $\frac{\partial{F_i}}{\partial p_i}>0$. Therefore, for $(\vec{q},\vec{p},\vec{x},\vec{y})$ in a sufficiently small neighborhood of $Z$, with $\vec{p}\neq 0$, if $\phi_t$ denotes the flow of $V$, then the $\vec{p}$ component of $\phi_t(\vec{q},\vec{p},\vec{x},\vec{y})$ has larger magnitude than $|\vec{p}|$, so such a point cannot be in the stable manifold of $Z$. Thus we conclude that near $Z$, $\Delta \subset \{(\vec{q},\vec{0},\vec{x},\vec{y})\}$. Therefore $T\Delta\subset (TZ)^{\perp_\omega}$. In a locally trivial chart of the bundle $\pi: \Delta \to Z$ the tangent space splits as $T_x\Delta=T_{\pi(x)}Z\oplus T_x\stab(\pi(x))$. Since the fibers are isotropic and symplectically orthogonal to the tangent space to $Z$, this implies $\Delta$ is isotropic. 
		\end{proof}

		\subsection{Front projections and tangential singularities} \label{s:front}
		
		The term ``front projection'' is often used in contact topology to refer specifically to the projection of the 1-jet space $\pi: J^1(N)\to N\times\R$. In particular, when $N=\R^n$, this gives the projection $\pi: (\R^{2n+1}_{(x_i,y_i,z)}, \ker(dz-\sum y_idx_i)) \to \R^{n+1}_{(x_i,z)}$. A Legendrian in $(\R^{2n+1}_{(x_i,y_i,z)}, \ker(dz-\sum y_idx_i))$ can be recovered from its front projection by the equations $y_i = \frac{\partial z}{\partial x_i}$. Observe that the fibers of this projection $\pi^{-1}(x_1,\cdots, x_n,z)$ are Legendrian (the tangent space is the span of the $\partial_{y_i}$).
		
		Another common ``front projection'' of a contact manifold is the projection of the unit co-tangent bundle $S^*M \to M$. Here, the fibers which are projected out are the co-tangent spheres, which again are Legendrian submanifolds foliating the total space. The Legendrian lift of a hypersurface representing a front projection in $M$ is the co-normal lift.
		
		It will be useful in this paper to use the term ``front projection'' to refer generally to the quotient of a contact manifold by a Legendrian foliation, or similarly the quotient of a symplectic manifold by a Lagrangian foliation (e.g. $T^*M\to M$ gives a fibration of $T^*M$ by the Lagrangian co-tangent fibers). We will essentially always be using the co-tangent front projections, but we will analyze these projections by analyzing the interactions of the leaves of the foliation with Lagrangian and Legendrian submanifolds of interest.
		
		A front projection of a Legendrian in the unit co-tangent bundle has tangential singularities, whenever the front projection fails to be an immersion. After a generic perturbation, one obtains an initial stratification of the Legendrian into submanifolds $\Sigma^k$ such that the front projection at a point $p\in \Sigma^k$ drops rank by $k$ where $\Sigma^k$ has codimension $k(k+1)/2$ in the Legendrian. The \emph{Thom-Boardman stratification} extends this more deeply, by inductively defining $\Sigma^{k_1,\cdots, k_{i-1},k_i}$ to be the set of points $p\in \Sigma^{k_1,\cdots, k_{i-1}}$ such that the restriction of the front projection to $\Sigma^{k_1,\cdots, k_{i-1}}$ drops rank by $k_i$ at $p$. For a generic Legendrian, each $\Sigma^{k_1,\cdots, k_i}$ is a smooth submanifold of the Legendrian of a predicted co-dimension. In particular once we fix the dimension of the Legendrian, there are a finite number of types $(k_1,\cdots, k_i,0)$ such that generically only $\Sigma^{k_1,\cdots, k_i,0}$ singularities occur (others have too large codimension).
		
		The simplest type of tangential singularities are the $\Sigma^{1,0}$ singularities. In a generic front projection of a smooth Legendrian $\cK$, the locus of $\Sigma^1$ singularities has co-dimension $1$ in $\cK$. Suppose $\pi: (Y,\xi)\to Z$ is a Legendrian front projection and $\cK\subset (Y,\xi)$ a Legendrian. Then near a point in the $\Sigma^{1,0}$ locus, there are coordinates $(x_1,\cdots, x_{n-2},x_{n-1})$ on $\cK$ and $(y_1,\cdots, y_n)$ on $Z$ such that $\pi(x_1,\cdots, x_{n-2},x_{n-1})= (x_1,\cdots, x_{n-2},x_{n-1}^2,x_{n-1}^3)$. The image looks like the product with $\R^{n-2}$ of the semi-cubical cusp.
		
		\subsection{Arboreal singularities}	\label{s:arb}
		
		To each tree (acyclic connected graph), Nadler \cite{N1} associates a topological stratified complex. If the tree has $N$ vertices, then the highest dimension of the strata is $N-1$ and the complex can be built from $N$ top dimensional strata with boundary and corners, glued together in a manner determined by the edges of the tree. Legendrian and Lagrangian models of this singularity are associated to the tree together with a choice of root vertex. The root induces a partial ordering on the vertices of the tree, and associates each vertex to a level given by its distance to the root. The Lagrangian model associated to the rooted tree $\mathcal{T}$ with $N$ vertices is the union of the zero section in $T^*\R^{N-1}$ with the positive co-normal to an \emph{arboreal hypersurface} in $\R^{N-1}$ associated to the rooted forest (disjoint union of trees) obtained by deleting the root of $\mathcal{T}$. An arboreal hypersurface for a rooted forest $\mathcal{F}$ with vertices $\{v_{\alpha}\}_{\alpha=1}^{N-1}$ is $C^0$ close to the following stratified subset of $\R^{N-1}$
		$$H_{\mathcal{F}}=\cup_{\alpha=1}^{N-1} P_{\alpha} \text{ where } P_{\alpha}=\{(x_1,\cdots,x_{N-1})\in \R^{N-1}\mid x_\alpha=0,\; x_\beta\geq 0 \; \forall v_{\beta}<v_{\alpha} \}.$$
		Each stratum $P_{\alpha}$ comes with a co-orientation by $\partial_{x_\alpha}$. The arboreal hypersurface that we actually use to take the co-normal of is the ``smoothed arboreal hypersurface'' which modifies $P_{\alpha}$ near its boundary so that the tangent spaces and co-orientations of $P_\alpha$ and $P_\beta$ agree along their intersection with $P_\alpha$ whenever $v_\beta<v_\alpha$. For the complete details of arboreal hypersurfaces see Section 3 of \cite{N1} or section 4.3 of \cite{N2}.
		
		For our global Weinstein set-up, we will require a signed version of Nadler's arboreal singularities/hypersurfaces. Therefore, we give a variation here on Nadler's arboreal hypersurface construction \cite[\S 3.2]{N2}. Let $\mathcal{T}$ be a rooted tree, such that all edges $e_{\hat\alpha,\alpha}$ which are not adjacent to the root are decorated by a sign $\sigma_{\hat\alpha,\alpha}\in \{\pm 1\}$ (where $e_{\hat\alpha,\alpha}$ denotes an edge from $v_{\hat\alpha}$ to $v_{\alpha}$). As above, let $\cF$ be the rooted forest induced by deleting the root of $\mathcal{T}$. Then the \emph{signed arboreal hypersurface} associated to the signed rooted forest $\cF$ is a $C^0$ close smoothing of the union of the strata
		$$P_{\alpha}=\{(x_1,\cdots,x_{N-1})\in \R^{N-1}\mid x_\alpha=0,\; \sigma_{\beta,\beta'} x_\beta\geq 0 \; \forall v_{\beta}<v_{\alpha} \text{ where } e_{\beta,\beta'} \text{ points towards } v_\alpha \}.$$
		The smoothing is performed by first choosing a function $f:\R^2\to \R$ which is a submersion and takes on positive, zero, and negative values as shown in figure \ref{fig:smoothing}.
		
		\begin{figure}
			\centering
			\includegraphics[scale=.5]{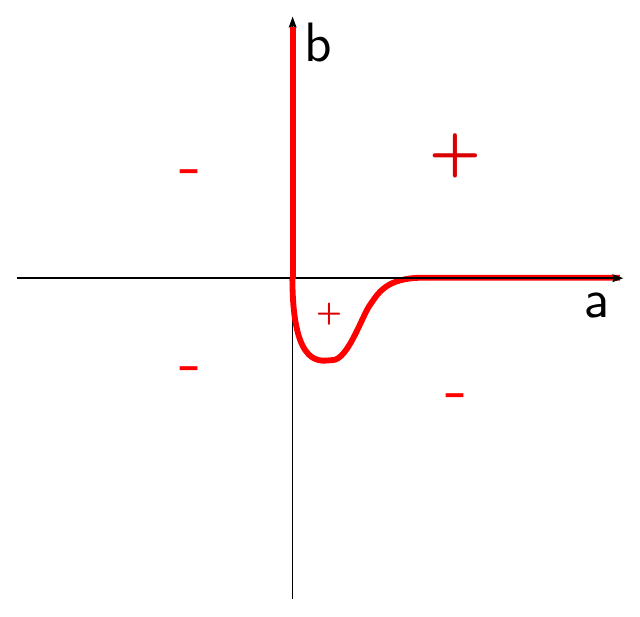}
			\caption{The thick red curve indicates the region where $f(a,b)=0$.}
			\label{fig:smoothing}
		\end{figure}
		
		Given a vertex $v_\alpha$, let $v_{\alpha_0}$ be the root vertex in the forest connected to $v_\alpha$ and let $v_{\alpha_0},v_{\alpha_1},\cdots, v_{\alpha_{k-1}},v_{\alpha_k}=v_\alpha$ be a directed chain from $v_{\alpha_0}$ to $v_{\alpha}$. Define 
		$$g_k^\alpha = \sigma_{\alpha_{k-1},\alpha_k}f(\sigma_{\alpha_{k-1},\alpha_k}x_{\alpha_{k-1}},\sigma_{\alpha_{k-1},\alpha_k}x_{\alpha_k}),$$ 
		and via downward induction define 
		$$g_j^\alpha = \sigma_{\alpha_{j-1},\alpha_j}f(\sigma_{\alpha_{j-1},\alpha_j} x_{j-1},\sigma_{\alpha_{j-1},\alpha_j}g_{j+1}).$$
		Set $H_\alpha = \{g_0^\alpha=0\}$ co-oriented by $\nabla g_0^\alpha$. Then the smoothed arboreal hypersurface is the union of the smoothed strata 
		$$\bigcup_{v_\alpha\in \cF} H_\alpha.$$

		The arboreal hypersurfaces associated to signed rooted graphs with two or three vertices are shown in figure \ref{fig:sgnA2} and \ref{fig:sgnA3}.
		

		\begin{figure}
			\centering
			\includegraphics[scale=.8]{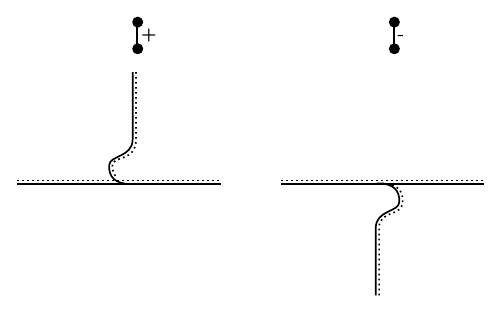}
			\caption{Arboreal hypersurfaces associated with signed rooted graphs with two vertices.}
			\label{fig:sgnA2}
		\end{figure}
		
		\begin{figure}
			\centering
			\includegraphics[scale=.8]{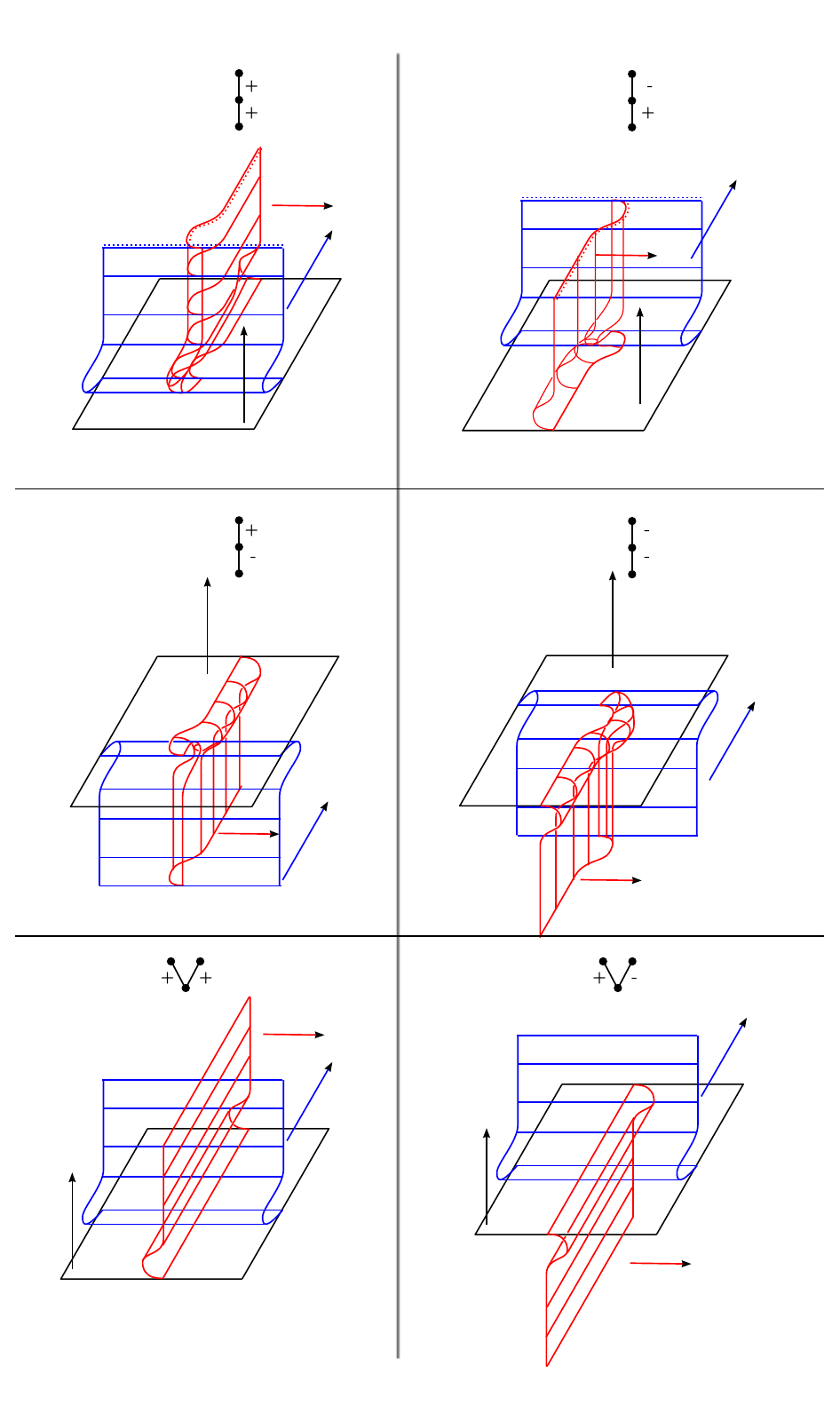}
			\caption{Arboreal hypersurfaces associated with signed rooted graphs with three vertices.}
			\label{fig:sgnA3}
		\end{figure}		
		
		A slight generalization of arboreal singularities allows for interactions between strata involving some strata with boundary. These are defined as \emph{generalized arboreal singularities} in \cite{N2}. A \emph{leafy rooted forest} is a rooted forest $\mathcal{F}$ together with a collection $\ell$ of marked leaves--vertices which are maximal with respect to the partial order induced by the roots. To a leafy forest $(\mathcal{F},\ell)$, define another rooted forest $\mathcal{F}^+$ which adds another vertex above each leaf in $\ell$. Then the generalized arboreal hypersurface associated to $(\mathcal{F},\ell)$ is a hypersurface associated to $\cF^+$ in $\R^{|\mathcal{F}^+|}$ given by a similarly defined smoothing of
		$$H_{(\mathcal{F},\ell)}=\bigcup_{v_{\alpha}\in \mathcal{F}^+\setminus \ell} P_{\alpha} $$
		The first Lagrangian generalized arboreal singularity is the zero section plus the co-normal to a hypersurface with boundary as shown in figure \ref{fig:genarb}.
		
		\begin{figure}
			\centering
			\includegraphics[scale=.4]{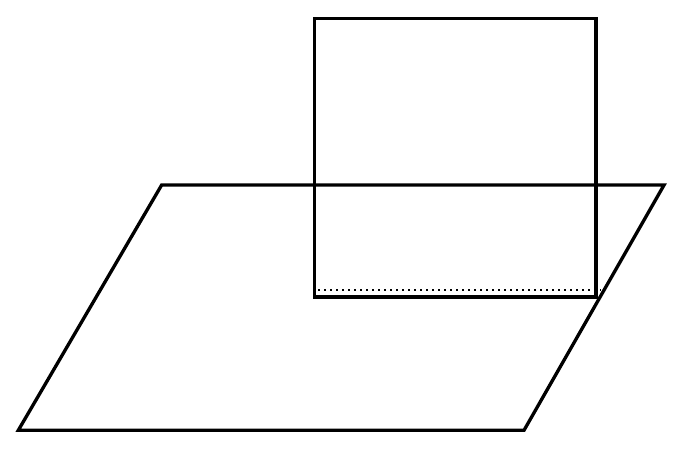}
			\caption{Lagrangian $2$-dimensional generalized arboreal singularity.}
			\label{fig:genarb}
		\end{figure}
		
		Arboreal hypersurfaces are defined up to an ambient isotopy of $\R^n$, and any such ambient isotopy extends to a symplectic isotopy of $T^*\R^n$ taking co-normals to co-normals.
		
		\begin{remark} Because a Lagrangian arboreal singularity is defined as the union of the zero section of $T^*\R^n$ together with the positive co-normals to an arboreal hypersurface (which is well defined up to ambient isotopy of $\R^n$), a symplectic neighborhood of a given Lagrangian arboreal singularity is determined up to symplectomorphism by the signed rooted tree.\end{remark}
		
		We formulate here a characterization to check whether a given singular hypersurface in $\R^n$ is an arboreal hypersurface, and similarly for a generalized arboreal hypersurface.
		
		\begin{prop} \label{p:arbhyp}
			A singular hypersurface $\mathcal{H}$ in $\R^n$ is a \emph{smoothed arboreal hypersurface} if and only if it can be written as the union of closed hypersurfaces with boundary and corners (of any codimension) $H_1,\cdots, H_j$, such that each point in $\mathcal{H}$ is in the interior of a unique $H_i$ and the following properties hold.
			\begin{enumerate}
				\item \label{c:an} If $p$ is a point in a codimension $k_i$ corner of $H_i$, then there exists a unique set $H_{j_0^i},\cdots , H_{j_{k_i-1}^i}$ such that $p$ is in a codimension $r$ corner of $H_{j_r^i}$ for $r=0,\cdots, k_i-1$. Moreover, at $p$ $H_{j_0^i},\cdots, H_{j_{k_i-1}^i}$ are tangent of order $k_i$ to $H_i$.
				\item \label{c:int} If $C_i$ is a codimension $k_i$ corner of $H_{j_i}$ for $i=1,\cdots, \ell$ and the sets of $H_{j_r^i}$ from condition \ref{c:an} are disjoint for distinct $C_i$, then for each $2\leq m\leq \ell$, $C_1\cap \cdots \cap C_{m-1}$ intersects $C_m$ transversally in an $n-k_1-\cdots - k_m-m$ submanifold of $\R^n$. In particular the intersection is empty if $n<k_1+\cdots +k_m+m$.
			\end{enumerate}
			
			A singular hypersurface $\mathcal{H}$ in $\R^n$ is a generalized arboreal hypersurface if it meets condition \ref{c:int}, but we loosen condition \ref{c:an} to the following:
			
			(1') If $p$ is a point in a codimension $k_i$ corner of $H_i$, then there exists a unique set $H_{j_0^i},\cdots , H_{j_{k_i-1}^i}$ such that $p$ is in a codimension $r$ corner of $H_{j_r^i}$ for $r=0,\cdots, k_i-2$. Moreover, at $p$ $H_{j_0^i},\cdots, H_{j_{k_i-2}^i}$ are tangent of order $k_i$ to $H_i$.
		\end{prop}
		
		A similarly straightforward statement about arboreal hypersurfaces is the following.
		
		\begin{prop}
			Suppose $H_i\subset \R^n$ are arboreal hypersurfaces for $i=1,\cdots, k$. Then there exist perturbations $\phi_i:\R^n\to \R^n$ arbitrarily close to the identity such that $\cup_{i=1}^k\phi_i(H_i)\subset \R^n$ is a hypersurface with only arboreal hypersurface singularities.
		\end{prop}
		
		\begin{remark} An arboreal hypersurface is the union of smooth (non-compact) hypersurfaces (codimension one). In contrast, a generic front projection of a Legendrian is Whitney stratified but need not be the union of smooth codimension one submanifolds.\end{remark}
		
		\subsubsection{Translating between arboreal Legendrians and Lagrangians}
		{
			An arboreal Legendrian submanifold of a contact manifold was defined in \cite{N1} to be a singular Legendrian such that near each singular point, there is a  local front projection whose image is an arboreal hypersurface. By definition, if an arboreal Legendrian of dimension $n-1$ lies in the unit co-tangent bundle $S^*M$ such that its corresponding front projection is an arboreal hypersurface, the union of the zero section with the co-normal cone of the Legendrian is an arboreal Lagrangian of dimension $n$. Conversely, the intersection of an arboreal skeleton of a Weinstein manifold with a contact type hypersurface yields an arboreal Legendrian submanifold. 
			
			Another way to get from an arboreal Legendrian singularity in a contact manifold to an arboreal Lagrangian singularity in a symplectic manifold is to take the Lagrangian projection. Namely, in a neighborhood of the singular point where the Reeb flow acts freely, mod out by the Reeb flow lines. The Legendrian condition implies that the image of each stratum in the arboreal Legendrian is mapped by an exact Lagrangian embedding into the Lagrangian projection.
		}
		
			\subsection{Weinstein homotopies near an invariant isotropic submanifold} \label{s:Whtpy}
			The primary mechanism to manipulate a Weinstein structure via a homotopy is to focus on its behavior restricted to an isotropic submanifold along which the Liouville vector field $V$ is tangent. The details of this are covered in \cite[Chapter 12]{CE}, and in this section we review the relevant results from that chapter and adapt them to the following statement which we will use repeatedly to manipulate the skeleton through Weinstein homotopies. The primary use of this proposition will be in the case that $\Delta$ is the stable manifold of a component of Morse-Bott* zeros. In this case, $\Delta$ satisfies the property that the Liouville vector field is repelling in the normal directions to $\Delta$. More precisely, if $(q,n)$ are local coordinates such that $\Delta=\{n=0\}$, then there exists $\varepsilon$ such that $\varepsilon \rho \leq V\cdot \rho \leq \varepsilon^{-1}\rho$. $\varepsilon$ depends on the minimal positive eigenvalue of the zeros in $\Delta$.
			
			\begin{prop}\label{p:isotropichtpy}
				Let $\Delta$ be an isotropic submanifold (possibly with boundary) of $(W,\omega,V,\phi)$ such that $V$ is tangent to $\Delta$ and $V$ is repelling in the normal directions to $\Delta$. Let $(Y_t,\psi_t)$ be a compactly supported 1-parameter family of Lyapunov pairs on $\Delta$ with $(Y_0,\psi_0)=(V|_{\Delta},\phi|_{\Delta})$. Suppose the eigenvalues of $Y_0$ have real part $<1$ (true if $\Delta$ is a stable manifold). Then there exists a homotopy of Weinstein structures $(W,\omega, V_t,\phi_t)$ such that
				\begin{itemize} 
					\item $(V_0,\phi_0)=(V,\phi)$
					\item $(V_1|_{\Delta},\phi_1|_{\Delta})=(\lambda Y_1,\psi_1)$ for some constant $0<\lambda\leq 1$
					\item $(V_t,\phi_t)=(V,\phi)$ outside a neighborhood $U$ of $\Delta$
					\item $V_t$ is non-vanishing in $U\setminus \Delta$.
					\item If $\Delta_1,\cdots, \Delta_m$ are other isotropic submanifolds (possibly with boundary) where $V$ is tangent to each $\Delta_\ell$ and each $\Delta_\ell$ intersects $\Delta$ cleanly along a smooth submanifold $\Sigma_\ell$ (possibly non-compact and/or with boundary) which is invariant under the flow of $V_t$, then we can ensure that in smaller neighborhood $\widetilde{U}\subset U$ of a compact $C\subset \Delta$, $V_t$ is tangent to $\Delta_\ell \cap \widetilde{U}$ for all $\ell=1,\cdots, m$.
				\end{itemize}
			\end{prop}
			
			\begin{proof}
				
				The core of this statement relies on \cite[Lemma 12.8]{CE}. This lemma considers the following construction on a neighborhood of the isotropic submanifold $\Delta$. For simplicity of notation, it is assumed that the symplectic normal bundle is trivial, which will suffice in our applications. Therefore a neighborhood $U$ of $\Delta$ can be identified symplectically with $T^*\Delta\times \C^\ell$ with coordinates $(p_i,q_i,x_j,y_j)$. 
				
				Under the clean intersection hypothesis, we can choose these coordinates $(p_i,q_i,x_j,y_j)$ so that $\Delta_\ell \cap \widetilde{U}$ is preserved by the negative flow of $\sum_i p_i\partial_{p_i}+\frac{1}{2}\sum_j x_j\partial_{x_j}+y_j\partial_{y_j}$. 
				
				Given any vector field $Y$ on $\Delta$ we obtain a Liouville vector field on $(U,\omega)$
				$$\hat{Y}=\sum_i p_i\partial_{p_i}+\frac{1}{2}\sum_j x_j\partial_{x_j}+y_j\partial_{y_j}-X_{H}$$
				where $X_{H}$ is the Hamiltonian vector field associated with the function $H = \langle \sum_i p_idq_i, Y(q)\rangle$. The restriction of $\hat{Y}$ to $\Delta$ agrees with $Y$. Given a function $\psi$ on $\Delta$, we get an associated function on $U$:
				$$\hat{\psi} = \psi(\vec{q})+|\vec{p}|^2+|\vec{x}|^2+|\vec{y}|^2.$$
				\begin{lemma}{\cite[Lemma 12.8]{CE}}
					Suppose that all eigenvalues at zeros of $Y$ have real part $<1$. Let $K\subset \Delta$ be compact. Then the pair $(\hat{Y},\hat{\psi})$ has the following properties.
					\begin{itemize}
						\item The zeros of $\hat{Y}$ coincide with the zeros of $Y$ and have equidimensional null and stable spaces.
						\item If $\rho(q,p,x,y)=|\vec{p}|^2+|\vec{x}|^2+|\vec{y}|^2$, then there exists some $\varepsilon>0$ such that $\hat{Y}\cdot \rho\geq \varepsilon\rho$.
						\item If $Y$ is gradient-like for $\psi$ then $\hat{Y}$ is gradientlike for $\hat{\psi}$ near $K$.
						\item Suppose $Y=X|_\Delta$ where $X$ is a vector field defined on a neighborhood of $\Delta$ which is gradient-like for a function $\phi$. Then under an identification of the neighborhood of $\Delta$ with $T^*\Delta\times \C^\ell$ which equates $E_p^+(X)=T^*_p\Delta\times \C^\ell$ at each zero of $Y$, $\hat{Y}$ is gradient-like for $\phi$.
					\end{itemize}
				\end{lemma}
				
				We will use this lemma to obtain Weinstein structures $(\hat{Y}_t,\hat{\psi}_t)$ defined in a neighborhood of $\Delta$ which agree with $(\lambda(t)Y_t,\psi_t)$ on $\Delta$ for some $\lambda:[0,1]\to (0,1]$. The purpose of the rescaling is to ensure the hypothesis that all eigenvalues at zeros of $Y$ have real part $<1$. Let $\lambda(t)=1/f(t)$ where $f(t)\geq 1$, $f(t)$ is equal to $1$ near $0$, and $f(t)\geq$ the largest eigenvalue of any zero of $Y_t$. Set $\lambda=\lambda(1)$.
				
				Now perform the construction above to construct for each $t\in[0,1]$, $(\hat{Y}_t,\hat{\psi}_t)$ from $(\lambda(t)Y_t,\psi_t)$ (note that $\lambda(t)Y_t$ has the same zeros and flow-lines as $Y_t$ and $\lambda(t)Y_t$ is still gradient-like for $\psi_t$). Observe that $\hat\psi$ has no critical points off of $\Delta$ and correspondingly $\hat{Y}_t$ is non-vanishing off of $\Delta$.
				
				Next we need to splice the Weinstein structures $(\hat{Y}_t,\hat{\psi}_t)$ defined on the neighborhood of $\Delta$ into $(V,\phi)$. This can be done using \cite[Lemma 12.10]{CE} after verifying the necessary hypotheses as follows. We have chosen our identification such that for each zero $q_0$ of $Y$, the $(q_0,p,x,y)$ fibers for a fixed $q_0$ are identified with the unstable manifold of $V$ at $(q_0,0,0,0)$. Therefore if $\rho(q,p,x,y)=|p|^2+|x|^2+|y|^2$, there is a constant $\varepsilon$ determined by the minimal non-vanishing eigenvalue of $V$ along its zero set, such that $\varepsilon \rho \leq V\cdot \rho \leq \varepsilon^{-1}\rho$ in a neighborhood $U$ of $\Delta$, by the normal repelling condition. Therefore by \cite[Lemma 12.10]{CE} there exist Weinstein structures $(\widetilde{V}_t,\widetilde{\phi}_t)$ on $W$ which agree with $(V,\phi)$ outside the neighborhood $U$ of $\Delta$, with $(\hat{Y}_t,\hat{\psi}_t)$ on a smaller neighborhood of $\Delta$, and which have no critical points in $U\setminus \Delta$.
				
				Therefore, the Weinstein structures $(\widetilde{V}_t,\widetilde{\phi}_t)$ have the following desired properties
				\begin{itemize} 
					\item $(\widetilde{V}_1|_{\Delta},\widetilde{\phi}_1|_{\Delta})=(\lambda Y_1,\psi_1)$ where $0<\lambda=\lambda(1)\leq 1$
					\item $(\widetilde{V}_t,\widetilde{\phi}_t)=(V,\phi)$ outside a neighborhood $U$ of $\Delta$
					\item $V_t$ is non-vanishing in $U\setminus \Delta$.
				\end{itemize}
				
				The only missing piece to conclude the proof, is that we need to pre-concatenate this family with a homotopy of Weinstein structures from $(V,\phi)$ to $(\widetilde{V}_0,\widetilde{\phi}_0)$.
				
				The construction of the Liouville vector field $\widetilde{V}_t$ in Lemma 12.10 is obtained by observing that $V-\hat{Y}_t=X_{H_t}$ for some Hamiltonian functions, choosing a cut-off function $f=f(\rho)$ supported on $U$, and setting $\widetilde{V}_t=V+X_{fH_t}$. Because $\hat{Y}_0$ is also gradient-like for $\phi$, it is verified in the proof of \cite[Lemma 12.9]{CE} that the straight line homotopy between $V$ and $\widetilde{V}_0$ gives a homotopy of Liouville vector fields which are all gradient-like for $\phi$. The space of Lyapunov functions for a fixed vector field $\widetilde{V}_0$ is convex, so we can then use a straight line homotopy between $\phi$ and $\widetilde{\phi}_0$ to get a Weinstein homotopy from $(\widetilde{V}_0,\phi)$ to $(\widetilde{V}_0,\widetilde{\phi}_0)$. Pre-concatentating these two straight line homotopies with $(\widetilde{V}_t,\widetilde{\phi}_t)$ completes the proof.
			\end{proof}

	\section{Lagrangian bones, joints, and Weinstein arboreal singularities} \label{s:bones}

		\subsection{Bones and thickening}
		
		\begin{definition} If $Z$ is a connected component of zeros of a Morse-Bott* Liouville vector field $V$ and $\stab(Z)$ is a Lagrangian submanifold, we say that $\Delta=\stab(Z)$ is a \emph{bone} of the skeleton. The corresponding zero set $Z\subset \stab(Z)$ is the \emph{marrow} of the bone $\Delta$. 
		\end{definition}
		
		Note that a bone has boundary if and only if its marrow $Z$ does, and it is non-compact (but has compact closure in $W$) if $\stab(Z)\setminus Z$ is non-empty. 
		
		\begin{definition}
			The points in a bone in its boundary $\partial (\stab(Z))=\stab(\partial Z)$ will be called the \emph{exterior boundary} of the bone $\stab(Z)$.
		\end{definition}
		
		
		\begin{definition}
			The \emph{index} of a bone is the dimension of the stable manifold of a single point in the marrow (which is independent of the point).
		\end{definition}
			 
		In a Lagrangian bone of index $k$ whose marrow is a disk $D^{n-k}$, the complement of the marrow $\Delta\setminus Z$ will be diffeomorphic to $S^{k-1}\times D^{n-k+1}$.

		Next we prove that every stable manifold of a component of zeros which is isotropic of dimension $<n$ can be thickened to a Lagrangian bone. Moreover, we have some control over how to perform this thickening which we will utilize in section \ref{s:tangential}.
		
		\begin{prop}\label{p:thicken}
			Let $Z_0$ be a $d$-dimensional connected component of index $k$ zeros of a Liouville vector field $V_0$ with stable manifold $\stab_{V_0}(Z_0)$, and unstable manifold $\unstab_{V_0}(Z_0)$ with trivial normal bundle in the Weinstein manifold $(W^{2n},\omega_0,V_0,\phi_0)$. Using any choice of 
			
			\begin{itemize}
			\item A submanifold $Z_1$ which is the image of an $n-k$ dimensional isotropic embedding of $Z_0\times D^{n-k-d}$ into $\unstab_{V_0}(Z_0)$ such that $Z_0\times\{0\}$ is identified with $Z_0$ and $TZ_1\cap (T\unstab_{V_0}(Z_0))^{\perp_{\omega_0}}=0$, and
			\item Lagrangian foliation $\cF$ of $\unstab_{V_0}(Z_0)$ in a neighborhood of $Z_1$ such that there is a unique point of $Z_1$ in each leaf, 
			\end{itemize}
			
			then there exist a Weinstein homotopy of $(W,\omega_0,V_0,\phi_0)$ to $(W,\omega_1,V_1,\phi_1)$ such that $V_1$ has index $k$ zeros along $Z_1$, $\stab_{V_1}(Z_1)$ is a Lagrangian bone, and the unstable manifolds of the points of $Z_1$ agree with the leaves of the foliation $\cF$ in a neighborhood of $Z_1$.
		\end{prop}
		
		\begin{cor}
			Every Weinstein manifold is Weinstein homotopic to one whose skeleton is the union of Lagrangian bones.
		\end{cor}
		
		\begin{proof}
						
			We will line up the specified submanifolds and foliation with a model on $T^*Z_0\times \C^{n-d}$, by choosing appropriate coordinates on a neighborhood $U$ of $Z_1$.
			
			First, we consider the submanifold $\unstab(Z_0)$ in the neighborhood $U$. This is a co-isotropic submanifold of dimension $2n-k$, so there is a $k$-dimensional isotropic sub-bundle $(T\unstab(Z_0))^{\perp_{\omega_0}}\subset T\unstab(Z_0)$. Choose coordinates $x_1,\cdots, x_k,y_1,\cdots, y_k$ on $U$ such that $\unstab(Z_0)=\{x_1=\cdots=x_k=0\}$, $(T\unstab(Z_0))^{\perp_{\omega_0}}$ is spanned by $(\partial_{y_1},\cdots, \partial_{y_k})$, $Z_1\subset \{y_1=\cdots=y_k=0\}$ (which can be arranged by the assumption $TZ_1\cap (T\unstab_{V_0}(Z_0))^{\perp_{\omega_0}}=0$). Observe that $\{x_1=\cdots=x_k=y_1=\cdots=y_k\}$ is a symplectic $2(n-k)$ dimensional sub-manifold (it is isomorphic to the symplectic reduction of $\unstab_{V_0}(Z_0)\cap U$), so we can arrange this choice of coordinates symplectically so that on $U$, $$\omega_0 = \sum_{i=1}^k dx_i\wedge dy_i + \omega_M$$ where $\omega_M=i_M^*\omega_0$.
			
			The leaves of the Lagrangian foliation $\cF$ on the co-isotropic submanifold $\unstab_{V_0}(Z_0)$ necessarily satisfy $$\spn(\partial_{y_1},\cdots, \partial_{y_k})=(T\unstab_{V_0}(Z_0))^{\perp_{\omega_0}}\subset T\cF_z.$$
			Therefore, $\cF \cap M$ gives a Lagrangian foliation of $M$ such that each leaf contains a unique point of $Z_1$. Since $Z_1$ is an isotropic $(n-k)$ dimensional submanifold of $M$, it is Lagrangian and has a standard neighborhood modeled on its co-tangent bundle, and we can identify the co-tangent fibers with the Lagrangian foliation $\cF\cap M$. We can choose the coordinates $q_j,p_j$ for this identification so that $Z_1=\{p_1=\cdots=p_{n-k}=0\}$, $Z_0 = \{q_{d+1}=\cdots=q_{n-k}=p_1=\cdots=p_{n-k}=0\}$, and the leaves of $\cF\cap M$ are $\{q_1=q_1^0,\cdots, q_{n-k}=q_{n-k}^0\}$.
			
			
			Then the result follows from the following lemma.
		\end{proof}
				
		\begin{lemma}\label{l:thicken}
			There exists a Weinstein homotopy from the standard radial Weinstein structure on $(\C^n,\sum_i dp_i\wedge dq_i)$ with a unique index zero critical point at the origin, to a Weinstein structure $(V_1,\phi_1)$ such that
			\begin{itemize} 
				\item The vanishing set $Z$ of $V_1$ is a Lagrangian disk in the zero section,
				\item $(V_1,\phi_1)$ agrees with the canonical co-tangent Liouville vector field on $T^*\R^n$ in a neighborhood of $Z$, 
				\item $(V_1,\phi_1)$ agrees with the standard radial structure outside a larger neighborhood $U$ of $Z$, and
				\item $Skel(\C^n,\omega_{std},V_1,\phi_1)=Z$.
			\end{itemize}
		\end{lemma}
		
		\begin{proof}
			
			Let $f(|q|^2)$ be a monotonically decreasing cut-off function which is $1$ when $|q|^2\leq \delta$ and vanishes for $|q|^2\geq 2\delta$. Consider the family of functions $\psi_t = (1-tf(|q|^2))|q|^2$ on the isotropic $(q_1,\cdots, q_n)$ plane. The gradient using the standard metric (rescaled by a factor of $\frac{1}{4}$) gives the following family of vector fields on the $(q_1,\cdots, q_n)$ plane
			$$Y_t = \frac{1}{2}\left(1-tf(|q|^2)-tf'(|q|^2)|q|^2 \right)\sum_i q_i \partial_{q_i}.$$ 
			Observe that $Y_t(\vec{q},\vec{p})=0$ if and only if $\vec{q}=0$ or $t=f(|q|^2)=1$. Additionally, $Y_0 = \frac{1}{2}\sum q_i\partial_{q_i}$ is the restriction of the standard radial Liouville vector field $\frac{1}{2}\sum p_i\partial_{p_i}+q_i\partial_{q_i}$ to the $(q_1,\cdots, q_n)$ plane. Also, $\psi_0 = |q|^2$ is the restriction of the radial function $|p|^2+|q|^2$ to the $(q_1,\cdots, q_n)$ plane. Note that the eigenvalues of this vector field at the origin are $\frac{1}{2}<1$. Apply Proposition \ref{p:isotropichtpy} to this family to get the desired result.

		\end{proof}

		\subsection{Holonomy and notions of genericity}
		
			Because we need each stable manifold component to be Lagrangian in order to have the strata of the skeleton determine its Weinstein neighborhood, we have to work with non-generic Weinstein structures $(V,\phi)$. However, we would like to be able to perturb this structure within the class of Weinstein structures with Lagrangian bones. The obvious perturbations which preserve the Morse-Bott* families of zeros are perturbations of $(V,\phi)$ supported away from the zero set of $V$.
			
			In a subdomain $U$ of a Weinstein manifold $(W,\omega, V,\phi)$ where $V$ is non-vanishing, the flow of $V$ determines a symplectic cobordism structure on $U$ from the part of $\partial U$ where $V$ points inward, $\partial_-U$, to the part of $\partial U$ where $V$ points outward, $\partial_+U$. $V$ determines a contact structure, $\xi_-$, on $\partial_-U$ and another, $\xi_+$, on $\partial_+U$ and the flow of $V$ gives a contactomorphism $\Gamma: (\partial_+U,\xi_+) \to (\partial_-,\xi_-)$ called the holonomy.
			
			By \cite[Lemma 12.5]{CE} the holonomy of a Weinstein cobordism can be modified by any contact isotopy by a homotopy of Liouville vector fields which remain gradient-like for the fixed Weinstein function.  We will call such Weinstein homotopies \emph{holonomy homotopies}.
			
			In particular, any isotopy of contact isotropic submanifolds can be realized by an ambient contact isotopy, which can then be incorporated into the holonomy. The notion of genericity which we will regularly employ is stability of the skeleton under small contact perturbations of the holonomy in domains where the Liouville vector field is non-vanishing.

		\subsection{Neighborhoods and intersections of Lagrangian bones}
			
			Because the stable manifolds of distinct zeros are disjoint, two bones $\Delta_i$ and $\Delta_j$ are also disjoint. However, their closures $\overline{\Delta_i}$ and $\overline{\Delta_j}$ may intersect non-trivially.
			
			Each bone $\Delta$ is a Lagrangian submanifold, possibly with boundary. Therefore it has a unique symplectic neighborhood modeled on its co-tangent bundle. The model for the Liouville vector field depends on the marrow and index of the bone. The marrow $Z$ of an index $k$ Lagrangian bone is $n-k$ dimensional so that the union of the $k$-dimensional stable manifolds of each of the points in the zero set is $n$ dimensional (Lagrangian). The bone $\Delta$ splits as $Z\times \R^k$ with local coordinates $q_i$ on $Z$ and $x_i$ on $\R^k$. The neighborhood of the bone splits as $T^*Z \times (\R^2)^{k}$ with local coordinates $(p_j,q_j)$ on $T^*Z$ and coordinates $(x_i,y_i)$ on $(\R^2)^k$ such that the Liouville vector field vanishes along $Z=\{x_i=y_i=p_j=0\}$, the stable manifold of $Z$ is $\Delta=\{y_i=p_j=0\}$, and the unstable manifold of $Z$ is $\{x_i=0\}$. The unstable manifold of each zero in the marrow of $\Delta$ is $n$-dimensional (spanned by the $\partial_{p_i}$ and $\partial_{y_i}$ directions), and the union of all of these unstable manifolds is $2n-k$ dimensional.
			
			
			Another bone, $\Delta_2$ may enter the neighborhood $N$ of the marrow of a bone $\Delta_1$ as a $V$-invariant submanifold. A neighborhood of the marrow of $\Delta_1$ will intersect $\Delta_2$ precisely when there is a Liouville flow line from the marrow of $\Delta_1$ to the marrow of $\Delta_2$, or equivalently when the stable manifold of $Z_2$ intersects the unstable manifold of $Z_1$.
			
			\begin{definition}
				We say $\Delta_2$ has a \emph{joint} $H$ on $\Delta_1$ if $H=\overline{\Delta_2}\cap \Delta_1$.
			\end{definition}
			
			In this case, the $V$-invariant bone $\Delta_2$ intersects $\partial N$ in the region where $V$ is outwardly transverse so $\partial N$ has positive contact type. This intersection $\Lc=\Delta_2\cap \partial N$ gives a Legendrian since $\Delta_2$ is $V$-invariant. 
			
			\begin{prop}\label{p:extbdry}
				A Weinstein manifold whose skeleton has Lagrangian bones is holonomy homotopic to a Weinstein manifold such that any joint $H$ of one bone $\Delta_2$ on another bone $\Delta_1$ is disjoint from the exterior boundary of $\Delta_1$.
			\end{prop}
			
			\begin{proof}
				
			Let $Z_1$ denote the marrow of $\Delta_1$. The unstable manifold of $\partial Z_1$ intersects the boundary $\partial N$ of a small neighborhood $N$ of $Z_1$ in the positive contact part of $\partial N$. $\unstab(\partial Z_1)\subset \partial N$ is a small open neighborhood of a contact isotropic embedding of $\partial Z_1$ (see figure \ref{fig:extbdry}).
			
			\begin{figure}
				\centering
				\includegraphics[scale=.4]{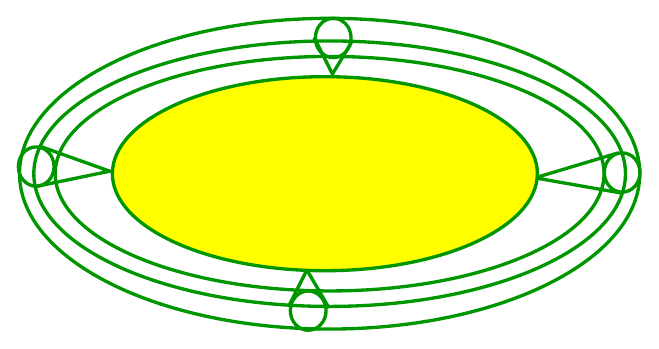}
				\caption{The unstable manifold of $\partial Z_1$ intersects the boundary of a neighborhood as $Z_1\times D^n$ where the core $Z_1\times\{0\}$ is contact isotropically embedded in $\partial N$.}
				\label{fig:extbdry}
			\end{figure}
						
			Within the positive contact portion of $\partial N$ we may perform a contact isotopy, and incorporate this into the holonomy by a Weinstein homotopy supported near $\partial N$. This allows us to perform a Legendrian isotopy of $\Lc=\Delta_2\cap \partial N$ so that it is disjoint from the unstable manifold of points in the exterior boundary $\partial Z_1$ which is a small neighborhood of a contact isotropic submanifold in $\partial N$.
			\end{proof}
			
			\begin{remark}
				This Proposition will often be applied without mention as we perform Weinstein homotopies modifying the skeleton. Particularly each time thickening is performed by applying Proposition \ref{p:thicken}, it may be necessary to subsequently use Proposition \ref{p:extbdry} to push the joints away from the exterior boundary of the newly thickened bones.
			\end{remark}
			
			The same type of holonomy homotopies allows us to perform small perturbations of the Legendrian to ensure various genericity assumptions. For example, by a contact isotopy perturbation, we may assume that $\Delta_2$ intersects $U^{2n-k}:=\unstab(Z_1)$ transversally in $N$. Let $\cK = \Delta_2 \cap U^{2n-k}\cap \partial N$. Then $\dim(\cK)=n-k-1$ and $\cK$ is an isotropic submanifold in the positive contact portion of $\partial N$. While $U^{2n-k}$ itself is co-isotropic, the slice where all $y_i=0$ (in the coordinates above) is a copy of $T^*Z_1$ (and is equivalent to the symplectic reduction of $U^{2n-k}$). We can ensure our coordinates are chosen so that $\cK$ is contained in this slice, since the transverse $\partial_{x_i}$ directions will be tangent to $\Delta_2$ near $\Delta_1$ since these stable flow lines in $\Delta_1$ glue to those in $\stab(Z_2)\cap \unstab(Z_1)$ to give unbroken flowlines in $\stab(Z_2)$. 
			
			Therefore viewing $\cK$ in the sphere cotangent bundle $S^*Z_1$, the map which takes a point $k\in \cK$ in the unstable manifold of the point $z\in Z_1$ to $z$ is a \emph{front projection}. Such a front projection of a Legendrian is generically an embedding away from a set of positive co-dimension in $\cK$. The image of the front projection is a hypersurface $J^{n-k-1}$ in $Z_1$ (which can be singular when the projection fails to be an embedding). Because $\Delta_2\cap N$ is $n$-dimensional and is invariant under backwards Liouville flow, and $\overline{\Delta_2}\cap Z_1 = J^{n-k-1}$, we must have $\overline{\Delta_2}\cap \Delta_1 = \stab(J^{n-k-1})$ (the union of the $k$-dimensional stable manifolds of each of the points in $J$). The joint $H^{n-1}=\stab(J^{n-k-1})=\overline{\Delta_2}\cap \Delta_1$ is therefore a hypersurface of $\Delta_1$ diffeomorphic to $J^{n-k-1}\times \R^k$. The unique $V$-invariant Lagrangian in $N$ with boundary on $H^{n-1}$ is described in terms of the splitting $N=T^*Z_1\times (\R^2)^k$ as the positive co-normal to $J^{n-k-1}$ in the $T^*Z_1$ factor times the stable $\R^k=\{y_i=0\}$ in the $(\R^2)^k$ factor. In particular, $\Delta_2\cap N$ lies in the slice where $\{y_i=0\}$, and is determined by the contact isotropic manifold $\cK^{n-k-1}\subset \partial N$ or equivalently its front projection to $Z_1^{n-k}$.
			
			We note and summarize a few immediate properties of joints between generic bones.
			
			\begin{lemma}[Properties of joints]\label{l:joints}
				\begin{enumerate}
					\item \label{jtn} If $\Delta_2$ has a non-empty joint on $\Delta_1$, the index of $\Delta_1$ must be strictly less than $n$.
					\item \label{jtord} If $\Delta_2$ has a non-empty joint on $\Delta_1$ then $\Delta_1$ cannot have a non-empty joint on $\Delta_2$.
					\item \label{jttr} If $\Delta_3$ has a non-empty joint on $\Delta_2$, and $\Delta_2$ has index $k>0$, then there exists a bone $\Delta_1$ such that both $\Delta_2$ and $\Delta_3$ have a non-empty joint on $\Delta_1$.
					\item \label{jtcon} For any bone $\Delta_1$ and compact subset $C\subset \Delta_1$, there is a neighborhood $U$ of $C$ and coordinates $(q,p)$ on $U$ identifying it with a subset of $T^*\Delta_1$ such that $Skel(W,\omega,V,\phi)\cap U$ is the union of the zero section with the co-oriented co-normal of the joints $H\cap C$ on $\Delta_1$.
				\end{enumerate}
			\end{lemma}
			
			Note that property \ref{jtcon} implies that every joint in a skeleton comes with a canonical co-orientation.
			
			\begin{proof}
				Statement \ref{jtn} follows because $\cK^{n-k-1}$ can only be the empty set if $k=n$. 
				
				Statement \ref{jtord} follows from the Weinstein condition. If $\Delta_2$ has a non-empty joint on $\Delta_1$, then the value of the Weinstein function $\phi$ on the marrow of $\Delta_2$, $\phi(Z_2)$ must be strictly greater than $\phi(Z_1)$.
				
				For statement \ref{jttr}, the fact that $\Delta_3$ has a non-empty joint on $\Delta_2$ implies that there is a point $z$ in the marrow of $\Delta_2$ whose unstable manifold intersects the stable manifold of $\Delta_3$. Since the index of $\Delta_2$ is greater than zero, the stable manifold of $z$ contains some point $p\neq z$. $p$ must be in the unstable manifold of some zero in the Weinstein manifold, which lies in the marrow of some bone $\Delta_1$. Then there is a broken flow-line from $\Delta_1$ through $\Delta_2$ to $\Delta_3$, which by gluing has a nearby unbroken flowline from $\Delta_1$ to $\Delta_3$.
				
				Statement \ref{jtcon} is equivalent to the statement above that a bone near its joint $H=J\times \R^k$ arises as the co-normal of $J$ times $\R^k_{(x_1,\cdots, x_k)}$ in $T^*Z_1\times (\R^2)^k$, which is the same as the co-normal to $H$ in $T^*(Z_1\times \R^k)$.
			\end{proof}
			
			\subsection{Classification of generic non-tangential singularities of skeleta} \label{s:genclass}
			
			We have seen that we can arrange that every two bones $\Delta_1$ and $\Delta_2$ meet (at most) along a joint $H$ which is a hypersurface on the interior of one of the bones, (say $\Delta_1$) such that the other bone $\Delta_2$ emanates as a co-normal Lagrangian in a neighborhood of the intersection. On the boundary of the neighborhood $N$ of $\Delta_1$, the other bone $\Delta_2$ intersects $\partial N$ as a Legendrian $\Lc$. The joint $H$ where $\Delta_1$ and $\Delta_2$ come together can be understood as the \emph{front projection} of $\Lc$ to $\Delta_1$. More specifically, we have a front projection to $Z_1$ of the intersection of $\cK^{n-k-1}=(\Lc\cap \unstab(Z_1))\subset S^*Z_1$, and then $\Lc \cong \cK\times \R^k$. Typically, the front projection of a Legendrian $\cK$ in $S^*Z_1$ has \emph{singularities of tangency} where tangent bundle $T\cK$ nontrivially intersects the tangent spaces to the sphere fiber leaves of the Legendrian foliation of $S^*Z_1$. These tangencies are places where the front projection fails to be an immersion, so the image hypersurface can develop singularities. We will address these singularities in section \ref{s:tangential}. For now, we show that these singularities of tangency are the only singularities we need to eliminate. 
			
			\begin{theorem}\label{thm:notang}
				Let $(W,\omega, V,\phi)$ be a Weinstein manifold whose skeleton has Lagrangian bones with joints disjoint from the exterior boundaries. Assume that for any pair of Lagrangian bones with non-empty joint, that the front projections of one bone onto another has no singularities of tangency. Then, after a generic holonomy perturbation, the skeleton has finitely many singularity types classified combinatorially by signed rooted trees with at most $n$ vertices corresponding to Nadler's arboreal singularities or by signed rooted trees with at most $n$ vertices with additional marked leaves corresponding to Nadler's generalized arboreal singularities.
			\end{theorem}
			
			\begin{proof}
			Any point $s$ in the skeleton lies on the interior of a unique bone $\Delta_1$ and if it is a singular point it lies on the joints of a collection of bones $\Delta_2,\cdots, \Delta_m$ with $\Delta_1$. Let $H_2,\cdots, H_m \subset \Delta_1$ be the joints corresponding to $\Delta_2,\cdots \Delta_m$. Then in a neighborhood $N$ of $\Delta_1$, $\Delta_j$ is the positive co-normal to $H_j$. It suffices to show that after a generic perturbation of holonomy in a neighborhood of the point $s\in \Delta_1$, the hypersurface $H_2\cup\cdots\cup H_m$ is arboreal.

			The arboreal type of the singularity is classified by a signed rooted tree with $m$ vertices, possibly with additional marked leaves. To determine this tree, we proceed by induction on $m$.
			
			If $m=1$, then either $s$ is in the interior of $\Delta_1$ and thus is a smooth point of the skeleton or on the exterior boundary of $\Delta_1$ and thus is a smooth boundary point of the skeleton. If it is on the interior, we associate the tree with a unique vertex (where the root is the unique vertex): $\bullet$. If $s$ is on the exterior boundary of $\Delta_1$ then we associate the tree with a unique vertex (being the root), and append a leaf to this vertex: $\bullet^{\lf}$.
			
			Now suppose $m>1$, and label by $\Delta_1$ the unique open bone containing $s$. Then $s$ is in the intersection of the joints $H_2,\cdots, H_m$ of $\Delta_j$ with $\Delta_1$ for $j=2,\cdots, m$. By assumption since the $H_j$ are disjoint from the exterior boundary of $\Delta_1$, $s$ is in the interior of $\Delta_1$. If $\Delta_1$ has index $k$, then the joints $H_j$ have the form $J_j\times \R^k$ where $J_j=H_j\cap Z_1$. By the product symmetry, it suffices to assume that $s\in Z_1$. Let $N$ be the standard neighborhood of $Z_1$. Then $(\Delta_2\cup \cdots \cup \Delta_m)\cap \partial_+ N$ gives a (possibly singular, possibly disconnected) Legendrian whose front projection on $\Delta_1$ is $H_2\cup \cdots \cup H_m$. The pre-image of $s$ under this front projection is a set of $\ell\leq m-1$ distinct points $t_1,\cdots, t_\ell$ corresponding to the $\ell\leq m-1$ distinct positive co-normal directions to the co-oriented hypersurfaces $H_2,\cdots H_m$.
			
			Now consider the skeleton in a neighborhood of one of these lifts $t_r$. There is a subset of bones $\{\Delta_{j_1^r},\cdots, \Delta_{j_{m_r}^r}\}\subset\{\Delta_2,\cdots, \Delta_m\}$ containing $t_r$ in their closure. This is because any bone which contains $t_r$ in its closure must contain the closure of the flow-line of $V$ through $t_r$ and thus must contain $s$. Therefore $m_r<m$, so by the inductive hypothesis, the skeleton has a (generalized) arboreal singularity at $t_r$ represented by a rooted (leafy) tree $T_r$.
			
			Now we claim that after a contact perturbation of the Liouville vector field supported in a neighborhood of $\Delta_1$, the singularity at $p$ is arboreally indexed by the signed, leafy rooted tree obtained from the disjoint union $T_1,\cdots, T_\ell$ of the trees associated to $t_1,\cdots, t_\ell$, such that the root of each of $T_1,\cdots, T_\ell$ is joined by an edge to a new root:
			
			$$\xymatrix{ T_1 & T_2 & \cdots & T_r\\
				& & \bullet \ar@{-}[ull] \ar@{-}[ur] \ar@{-}[u] \ar@{-}[ul]& }$$
			
			The inductive hypothesis does not label the edges adjacent to the roots of the $T_i$ by signs. Instead, we determine these signs from the relevant co-oriented hypersurfaces in $\Delta_1$. Each such edge corresponds to two bones $\Delta_j$ and $\Delta_{j'}$ whose joints on $\Delta_1$ are as follows. The joint $H_j$ of $\Delta_j$ with $\Delta_1$ is a smooth co-oriented hypersurface near $s$. The closure of the joint $H_{j'}$ of $\Delta_{j'}$ on $\Delta_1$ is a hypersurface with boundary, such that near $s$, $H_j\cap \overline{H_{j'}} = \partial \overline{H_{j'}}$ and $T_p(H_j)=T_p(\overline{H_{j'}})$ for $p\in \partial \overline{H_{j'}}$. The hypersurface $H_j\subset \Delta_1$ is locally separating and co-oriented. If $H_{j'}$ lies on the side of $H_j$ compatible with the co-orientation, let $\sigma_{j,j'}=+1$. If $H_{j'}$ lies on the opposite side of $H_j$ than indicated by the co-orientation, let $\sigma_{j,j'}=-1$. The models for these hypersurfaces are products with $\R^M$ of the models in figure \ref{fig:sgnA2}.

			Let $N$ be a small neighborhood of $\Delta_1$, and let $\cK_r$ denote the local Legendrian near $t_r$ in $\partial N\cap Skel(W,\omega, V,\phi)$. Because $V$ is non-vanishing and outward pointing to $\partial N$ along $\cK_r$, the skeleton near $t_r$ is modeled locally by $\cK_r\times \R$ in the symplectization $\partial N\times \R$. Since the skeleton is arboreal near $t_r$, and $\partial N$ intersects each stratum of the skeleton transversally near $t_r$, the Legendrian $\cK_r$ must also have only arboreal singularities. In other words, the Lagrangian projection of $\cK_r$ obtained by locally modding out by the Reeb flow has only arboreal singularities, and a generic front projection of $\cK_r$ to $H_{j_1^r}\cup\cdots \cup H_{j_{m_r}^r}$ is an arboreal hypersurface. We can ensure that the joints $H_{j_1^r}\cup \cdots \cup H_{j_{m_r}^r}$ are a sufficiently generic front projection by a holonomy homotopy supported near $t_i$.
						
			We may also perform holonomy homotopies supported near the points $t_1,\cdots, t_\ell$, to independently isotope the contact neighborhoods of $t_1,\cdots, t_\ell$ so that the strata of the arboreal hypersurfaces in the $T_r$ projection generically intersect the strata of the arboreal hypersurfaces in the $T_{r'}$ projection. Since the union of generically intersecting arboreal hypersurfaces is an arboreal hypersurface, $s$ is an arboreal singularity of the skeleton.
			\end{proof}
			

			\subsection{Arboreal singularities as skeleta} \label{s:model}
			
			Conversely, we can give a model showing that any arboreal singularity associated to a signed rooted tree arises in a Weinstein skeleton. Each vertex of $\mathcal{T}$ corresponds to a distinct Lagrangian bone. The root corresponds to the bone which contains the point on its interior. If the root bone has index $k$, then a vertex at height $\ell$ (distance $\ell$ from the root) corresponds to an index $k+\ell$ bone. If two vertices are connected by a path of upwardly oriented edges then the corresponding bones intersect along a non-empty joint. Here we construct model Weinstein structures exhibiting each arboreal singularity.
			
			\begin{theorem}\label{thm:model}
				Let $\mathcal{T}$ be a signed rooted tree with $n+1$ vertices. There exists a Weinstein structure on $(\C^{n}, \omega_{std})$, $(V,\phi)$ whose skeleton consists of $n+1$ Lagrangian bones meeting at an arboreal $T$ singularity. 
				
			\end{theorem}

			In general, given a rooted tree $\mathcal{T}$ with $n$ vertices, a $2m$ dimensional model for a type $\mathcal{T}$ arboreal skeleton where $n\leq m$ will be given by taking the product of the $2n$ dimensional model with $m-n$ copies of $\C$ with the standard $T^*\R$ Morse-Bott Weinstein structure. 
			
			\begin{proof}
				Let $X^\pm$ be a Liouville vector field on $\R^2$ which is gradient-like for a function $f^\pm$, such that $X^\pm$ has a cancelling pair of critical points $(0,\pm 2)$ of index $0$ and $(0,\pm 1)$ of index $1$ connected by a flow line $\gamma(t)=(0,\pm(2-t))$, and such that $(X^\pm,f^\pm)$ agrees with $(\frac{1}{2}x\partial_x+\frac{1}{2}y\partial_y, x^2+y^2)$ outside a neighborhood of $\gamma$ (see figure \ref{fig:CancPair}). We may arrange that the $\partial_x$ direction is in the unstable manifold of both critical points, and that $\varepsilon x^2 \leq xdx(X^\pm) \leq \frac{1}{\varepsilon} x^2$ in a neighborhood of $\gamma$.
				
				\begin{figure}
					\centering
					\includegraphics[scale=.8]{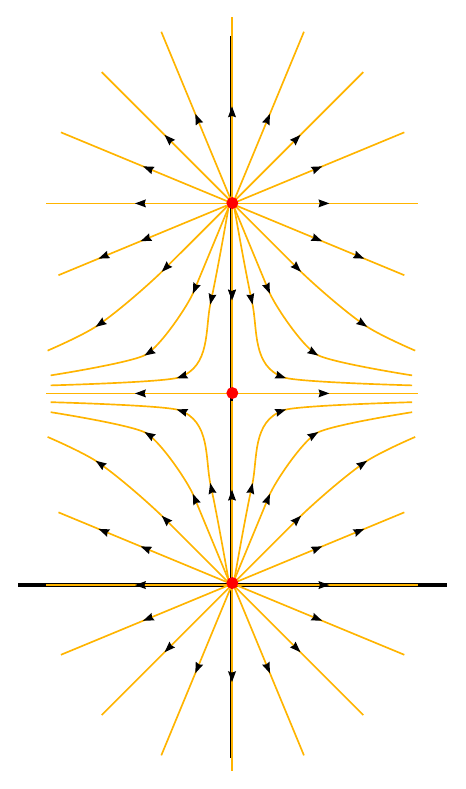} \hspace{1cm} \includegraphics[scale=.8]{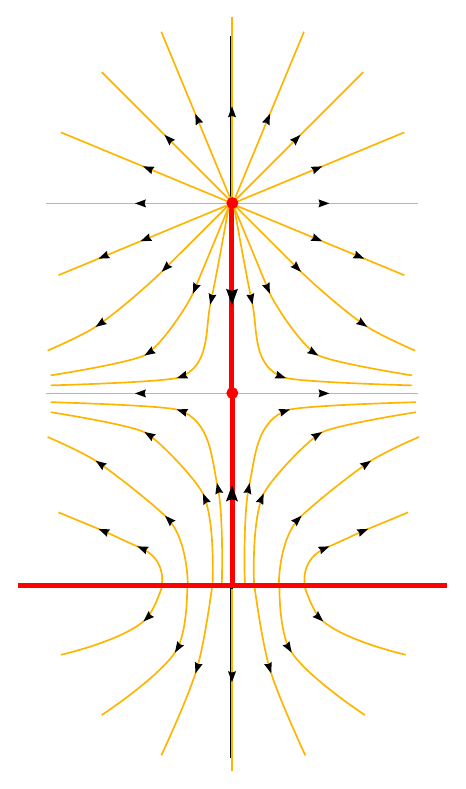}
					\caption{Left: flow of $X^+$. Right: flow of $X^+$ spliced with flow of $Y$. (Reflect over the $x$-axis for the $-$ versions.)}
					\label{fig:CancPair}
				\end{figure}
				
				Let $(Y,g)$ denote the canonical Weinstein structure on $\R^2$ viewed as $T^*\R$: $Y=y\partial_y$, $g(x,y)=y^2$.
				
				Let $\mathcal{T}'$ be signed the rooted forest obtained by deleting the root vertex of $\mathcal{T}$. Label the vertices of $\mathcal{T}'$ $v_1,\cdots, v_n$. For $i=1,\cdots, n$, define the Weinstein structure $\mathcal{W}_i=((\R^2)^n, \sum_i dx_i\wedge dy_i, \mathbf{V}_i, \mathbf{\phi}_i)$ by setting
				$$(\mathbf{V}_i,\Phi_i)=(V_1,\phi_1)\times \cdots (V_n,\phi_n)$$
				where 
				$$(V_j,\phi_j)=\begin{cases}
				(X^{\sigma_{\hat{j},j}},f^{\sigma_{\hat{j},j}}) & \text{ if } v_j\leq v_i\\
				(Y,g) & \text{ otherwise.}
				\end{cases}$$
				Let $(V_0,\phi_0)=(\sum_i y_i\partial_{y_i}, \sum_i y_i^2)$ be the standard cotangent Liouville structure, which we will associate to the root vertex $v_0$ of $\mathcal{T}$.
				
				Let 
				$$L_i = \{(x_1,y_1,\cdots, x_n,y_n) \mid \begin{cases}
				x_j=0, \frac{1}{2}\leq \sigma_{\hat{j},j} y_j\leq 2 &\text{ if } v_j\leq v_i\\
				y_j=0 & \text{ otherwise} 
				\end{cases} \}$$
				and $L_0=\{(x_1,y_1,\cdots, x_n,y_n)\mid y_j=0\; \forall j \}$.Then $L_i$ is a Lagrangian contained in the skeleton of $\mathcal{W}_i$. 
				
				Near each $L_i$ we can splice in the Weinstein structures $\mathcal{W}_i$ into the standard radial Weinstein structure on $\R^{2n}=T^*\R^n$, $(V_{rad}=\sum_i \frac{1}{2}x_i\partial_{x_i}+\frac{1}{2}y_i\partial_{y_i}, \phi_{rad}=\sum_i x_i^2+ y_i^2)$. 
				
				Let $U_i$ be a small regular neighborhood of $\widetilde{L}_i$ where
				$$\widetilde{L}_i = \{(x_1,y_1,\cdots, x_n,y_n) \mid \begin{cases}
				x_j=0, \frac{3}{4}\leq y_j\leq 2 &\text{ if } v_j\leq v_i\\
				y_j=0 & \text{ otherwise} 
				\end{cases} \}$$
				and let $U_0$ be a small regular neighborhood of $L_0=\widetilde{L}_0$. Then for $i\neq j$, by choosing $U_i$ sufficiently small, we may assume $U_i\cap U_j=\emptyset$. Then using these neighborhoods we can splice in the Weinstein structure $\mathcal{W}_i$ into a neighborhood of $\widetilde{L}_i$ contained in $U_i$, such that there are no critical points in $U_i\setminus L_i$. This splicing is achieved by \cite[Lemma 12.10]{CE}, because this set-up satisfies the hypotheses of that lemma, namely that $V_{rad}$ and $\mathbf{V}_i$ are both tangent to $L_i$ and we have estimates
				$$\varepsilon \rho_i \leq V_{rad}\cdot \rho_i, \mathbf{V}_i\cdot \rho_i \leq \varepsilon^{-1}\rho_i.$$
				where $\rho_i:\R^{2n}\to \R$ are the functions
				$$\rho_i(x_1,y_1,\cdots, x_n,y_n)=\sum_{v_j\leq v_i}x_j^2+\sum_{v_j\nleq v_i}y_j^2 $$
				and $\rho_0=\sum_i y_i^2$.
				
				After splicing each $\mathcal{W}_i$ in, we have a Weinstein structure $\mathcal{W}_{\mathcal{T}}$ which agrees with $\mathcal{W}_i$ in a neighborhood of each $\widetilde{L}_i$, and whose critical points are all contained in $L_i$. Therefore the skeleton consists of the union of the stable manifolds of the zeros of $\mathcal{W}_i$ for $i=0,\cdots, n$.
				
				Now we verify that after a holonomy perturbation, the resulting spliced Weinstein structure $\mathcal{W}_{\mathcal{T}}$ has an arboreal $\mathcal{T}$ singularity. Suppose $v_k<v_\ell$ and $v_k$ and $v_\ell$ are connected by an edge. Then $(\mathbf{V}_k,\Phi_k)$ and $(\mathbf{V}_\ell,\Phi_\ell)$ differ only in the $(x_\ell,y_\ell)$ coordinates. In the $(x_\ell,y_\ell)$ coordinate plane, the two spliced pieces meet in an arboreal $A_2$ singularity as in figure \ref{fig:CancPair}. For vertices $v_i<v_j$ that are connected by a chain of edges, there will be a joint of the bone in the $(\mathbf{V}_j,\Phi_j)$ skeleton onto the bone in the $(\mathbf{V}_i,\Phi_i)$ skeleton. However, without any holonomy perturbations, whenever $v_i$ and $v_j$ are separated by more than one edge, there will be a degenerate front projection that needs to be perturbed. 
				
				The simplest example is the linear tree with three vertices, where the root is an extremal vertex. In this case, the three stable manifolds forming the Lagrangian skeleton in $\C^2$ near the origin will be
				$$\Delta_0 = \{(x_1,0,x_2,0)\}$$
				$$\Delta_1= \{(0,y_1,x_2,0) \mid 0\leq y_1 \}$$
				$$\Delta_2 = \{(0,y_1,0,y_2) \mid 0\leq y_1,y_2 \}$$
				The joint between $\Delta_0$ and $\Delta_2$ is degenerate because it is a single point $(0,0)$ (it is the projection to the $(x_1,x_2)$ coordinates of $(0,y_1,0,y_2)$). Note that the joint between $\Delta_0$ and $\Delta_1$ $\{(0,0,x_2,0)\}$ and the joint between $\Delta_1$ and $\Delta_2$ $\{(0,y_1,0,0)\}$ are immersed hypersurfaces, so these do not require perturbation. An explicit holonomy perturbation to make the joint between $\Delta_0$ and $\Delta_2$ generic can be performed along the contact hypersurface $Y_1=\{y_1=\frac{1}{2}\}$. We will give a contact isotopy which fixes $\Delta_1\cap Y_1$, but perturbs $\Delta_2\cap Y_1$ so that it projects to $\Delta_0$ via an immersion. The diffeomorphism
				$$\phi(x_1,\frac{1}{2},x_2,y_2) = (x_1-ty_2^2,\frac{1}{2},x_2+ty_2,y_2)$$
				is a strict contactomorphism:
				$$\phi^*(\frac{1}{2}dx_1+y_2dx_2) = \frac{1}{2}dx_1-ty_2dy_2+y_2dx_2+ty_2dy_2 = \frac{1}{2}dx_1+y_2dx_2.$$
				Furthermore, $\phi$ fixes $\Delta_1\cap Y_1$:
				$$\phi(0,\frac{1}{2},x_2,0) = (0,\frac{1}{2},x_2,0)$$
				and modifies $\Delta_2\cap Y_2$ to
				$$\phi(0,\frac{1}{2},0,y_2) = (-ty_2^2,\frac{1}{2},ty_2,y_2)$$
				which projects to $\Delta_0$ as
				$$(-ty_2^2,\frac{1}{2},ty_2,y_2) \mapsto (-ty_2^2,ty_2)$$
				which is an immersion for any $t>0$. Therefore by incorporating this contactomorphism into the holonomy in a small neighborhood of $\Delta_0$, we end with an arboreal model for the tree with three stacked vertices.
				
				For the general case, we will have many pairs of vertices separated by more than one edge. For each such pair, we will have an associated contact perturbation, generalizing the above example as follows.
				
				Suppose we are building the model for the arboreal $\mathcal{T}$ singularity where $|\mathcal{T}|=n+1$, and we are trying to perturb the joint between $v_i$ and $v_j$ where $v_i<v_j$. Let $v_i<v_{k_1}<\cdots < v_{k_{m-1}} <v_{k_m}=v_j$ be the chain of vertices connecting $v_i$ to $v_j$ in the tree $\mathcal{T}$. Then the contactomorphism $\phi_{ij}: Y_i=\{y_i=c_i\}\to Y_i$ will be defined such that the $(x_\ell,y_\ell)$ components are given as
				$$\phi_{ij}^\ell(x_1,y_1,\cdots, x_n,y_n) = \begin{cases}
																(x_i-\frac{t_{ij}}{2c_i}(y_{k_1}^2+\cdots+y_{k_m}^2),y_i=c_i) & \text{ if } \ell = i\\
																(x_{k_s}+t_{ij}y_{k_s},y_{k_s}) & \text{ if } \ell = k_s\\
																(x_\ell,y_\ell) & \text{ if } \ell \notin\{i,k_1,\cdots, k_{m-1},k_m=j\}
																\end{cases}$$
				
				Observe that since $\Delta_i$ and $\Delta_j$ only differ in the coordinates $(x_\ell,y_\ell)$ for $\ell\in \{i,k_1,\cdots, k_{m-1},k_m=j\}$, the projection of $\Delta_j\cap Y_i$ to $\Delta_i$ will be an immersion since 
				$$\phi(0,c_i,0,y_{k_1},\cdots, 0,y_{k_{m-1}},0,y_{k_m=j}) = (-\frac{t_{ij}}{2c_i}(y_{k_1}^2+\cdots+y_{k_m}^2),c_i,t_{ij}y_{k_1},y_{k_1},\cdots, t_{ij}y_{k_m},y_{k_m})$$
				projects to $\Delta_1$ as the immersion
				$$(-\frac{t_{ij}}{2c_i}(y_{k_1}^2+\cdots+y_{k_m}^2),t_{ij}y_{k_1},\cdots, t_{ij}y_{k_m}).$$
				
				These perturbations should be incorporated into the holonomy in an order such that the highest joints are perturbed first (so that the descending skeleton to the lower bones is actually arboreal), and such that the perturbations of the higher joints are made with sufficiently small coefficients so as to not interfere with the perturbations of the lower joints. (In fact, based on explicit calculations for somewhat larger examples than the one above, it seems that the signs work out such that the higher perturbations actually only help the lower perturbations to be immersions, but for brevity of proof, we will utilize the fact that our perturbations can be made arbitrarily small since any $t_{ij}>0$ makes the above projection from $\Delta_j\cap Y_i$ to $\Delta_i$ an immersion.) Therefore the coefficients $t_{ij}$ should be chosen relatively small compared to $t_{i'j}$ for $i'<i$. Then because the contactomorphism associated to the pair $(i,j)$ is relatively small, its associated holonomy perturbation will keep $\Delta_j$ relatively $C^1$ close to the original (degenerately projecting) $\Delta_j$, so since the perturbation $\phi_{i'j}$ causes the original degenerately projecting $\Delta_j$ to project with no tangencies to $\Delta_{i'}$, it will also cause the $\phi_{ij}$ holonomy perturbed version of $\Delta_j$ to project with no tangencies to $\Delta_{i'}$.

			\end{proof}

\section{Singularities of tangency} \label{s:tangential}
	Each bone $\Delta$ contains a singular hypersurface worth of joints, such that given a neighborhood of $\Delta$ in $W$, the joints are the front projections of (singular) Legendrians in the positive contact type part of the boundary of the neighborhood. The front projection itself is specified by a Legendrian foliation of this contact boundary of the neighborhood, or equivalently a Lagrangian foliation of the neighborhood itself (defined away from the boundary and non-compact ends of $\Delta$).
	
	A front projection of a Legendrian in the unit co-tangent bundle may develop singularities, whenever the front projection fails to be an immersion. Equivalently, this occurs when the tangent space to the Legendrian nontrivially intersects the tangent space to a leaf of the foliation by cotangent spheres. The singularities that develop in the skeleton along a joint at a tangential singularity, fall outside the range of arboreal singularities and in sufficiently high dimensions realize infinitely many different singularity types. Therefore we would like to eliminate these tangential singularities from occurring in our skeleton.
	
	The Weinstein function $\phi$ takes a constant value $\phi(Z)$ on the marrow $Z$ of a bone $\Delta$, which we will call the \emph{$\phi$-value of $\Delta$}. We will address singularities of tangency of such front projections one bone at a time, starting with the joints lying in the highest $\phi$ valued bones, and ending with the joints lying in the lowest $\phi$ valued bones. If two bones have the same value $\phi(Z)$ we can deal with them in either order independently, or we can assume by genericity that each bone has a distinct value of $\phi(Z)$. Because the non-empty joints in the highest $\phi$-valued bones are necessarily the front projections of smooth Legendrians (possibly with boundary), performing the arborealization from the top to the bottom allows us to assume that the singularities of the Legendrian we are front projecting are arboreal. (To check that a Legendrian has arboreal singularities using the Lagrangian definitions of arboreal singularities, verify that the Lagrangian projection obtained by locally quotienting by the Reeb flow has only arboreal singularities.)
	
	Arborealizing singularities of tangency in the joints comes in two parts. As we go through the arborealizing procedure, we will take care not to introduce new singularities of tangency in higher $\phi$-valued bones, so the process terminates when we get to the lowest $\phi$-valued bone. To accomplish this, in the first part, we localize the flow around the singularities of tangency occurring at the joints on $\Delta$. The localization stage will increase the number of bones, by breaking up the existing bones into multiple bones. At first the skeleton will stay set-wise constant, but the stratification will change. Then, to keep all bones Lagrangian so that after a generic perturbation the only non-arboreal singularities come from singularities of tangency, the skeleton will grow some fins. 
	
	In the second part, we will remove the simplest type of singularities of tangency in these confined regions by breaking up the Lagrangian bones into more pieces which meet arboreally. The key idea is to allow the tangent bundle to the Legendrian to have discontinuities at arboreal singularities to skip tangencies to the foliation.

	\subsection{Localizing tangential singularities} \label{s:localize}
	
		Here we give the procedure to localize the Liouville flow into a neighborhood of a joint. As mentioned above, we perform this procedure as well as the procedure to eliminate $\Sigma^{1,0}$ singularities of tangency one bone at a time, starting with joints lying in the highest $\phi$ valued bones and ending with joints lying in the lowest $\phi$ valued bones. Fixing a bone $\Delta$ containing a non-empty joint, we modify the Weinstein structure on a neighborhood of $\Delta$ first using the following proposition.
	
		\begin{prop}\label{p:localize}
			Let $\Delta^n$ be an index $k$ Lagrangian bone with marrow $Z^{n-k}$ for $(W,\omega_0,V_0,\phi_0)$. Then there exist closed neighborhoods $\nu\Supset\nu_0 \Supset \nu_1 \Supset \widetilde{\nu} \supset Z$ so that the only zeros of $V_0$ in $\nu$ are in $Z$ and a Weinstein homotopic structure $(\omega_1,V_1,\phi_1)$ on $W$ such that
			\begin{itemize} 
				\item $Skel(W,\omega_1,V_1,\phi_1) \cap \nu_0$ is invariant under the positive flow of $V_1$,
				\item $\unstab(Z)\cap Skel(W,\omega_0,V_0,\phi_0)\cap \partial \nu_1$ is invariant under the positive flow of $V_1$ and attracts under this positive flow $Skel(W,\omega_1,V_1,\phi_1)\cap \nu_0\setminus (\partial \nu_0\cup \Delta)$,
				\item $Skel(W,\omega_1,V_1,\phi_1)$ agrees with $Skel(W,\omega_0,V_0,\phi_0)$ outside of $\nu$ and inside of $\widetilde{\nu}$,
				\item $Skel(W,\omega_1,V_1,\phi_1)$ and $Skel(W,\omega_0,V_0,\phi_0)$ are ambiently isotopic and thus have the same singularities.

			\end{itemize}
		\end{prop}
		
		\begin{proof}
			Choosing a sufficiently small neighborhood $\nu$ of $Z$, we can choose coordinates which identify it with a neighborhood of the zero section in $T^*Z\times \C^k$ such that for each $z\in Z$, $\stab(z) = \{z\}\times \R^k$ where $\R^k$ has coordinates $(x_1,\cdots, x_k)$. 
			
			Consider $\Lc:=\unstab_{V_0}(Z)\cap Skel(W,\omega_0,V_0,\phi_0)\cap \nu$, a potentially singular isotropic of dimension $n-k$. The unstable manifold $\unstab_{V_0}(Z)$ itself is co-isotropic, but contains a slice identified with $T^*Z\times\{0\}$ isomorphic to its symplectic reduction. Moreover, $\Lc$ is contained in this slice because it is part of the stable manifold of some other zero set which intersects $\nu$ as $\Lc\times \R^k_{(x_1,\cdots, x_k)}$. Since the stable manifold is isotropic and contains $\langle \partial_{x_1},\cdots, \partial_{x_k}\rangle$ in its tangent space, it must be contained in the slice $\{y_1=\cdots=y_k=0\}$. In $T^*Z\times \{0\}$, $\Lc$ is the co-oriented co-normal to a hypersurface $\cJ\subset Z$. The joints in $\Delta=Z\times D^k$ are given by $\cJ\times D^k$, and $Skel(W,\omega_0,V_0,\phi_0)\cap \nu$ is $Z\times D^k \cup \Lc\times D^k$. Along $\partial \nu$, $V_0$ points inward along the part of the skeleton identified with $(\Lc\cup Z)\times S^{k-1}$ and outward along $\partial \Lc \times D_{\varepsilon}^k$. Because $V_0$ is tangent to $\Lc$ and non-vanishing on $\Lc\setminus \cJ$, integrating its flow allows us to identify $\Lc\setminus \cJ$ with $(0,1)\times \cK$ where $\cK$ is a potentially singular contact isotropic submanifold of the positive contact part of $\partial \nu$. Let $s$ denote the coordinate on $(0,1)$. It will be useful to consider $\cK$ as a Legendrian in the co-sphere bundle $S^*Z\times \{0\}$. Because of the ordering we consider the bones $\Delta$, we may assume that if $\cK$ has singularities, they are arboreal.
			
			$\cK$ decomposes into smooth Legendrian strata $\cK_j$ based on which bone $\Delta_j$ contains that stratum. Each such bone $\Delta_j$ has a non-empty joint on $\Delta$ given as $\cJ_j\times D^k\subset Z\times \R^k$. Since $\Delta_j$ is locally the co-normal to its joint, $\Delta_j\subset T^*Z\times \R^k$.
			
			Choose a Morse-Bott* function $f_j: \cK_j \to \R$ which has connected critical locus a manifold (possibly with boundary) $C\subset \cK_j$ such that $\stab_{\nabla f_j}C=\cK_j$. We will modify the Weinstein structure to create canceling pairs of zeros along $\{\frac{1}{3}\}\times C$ and $\{\frac{2}{3}\}\times C$ whose local stable manifolds are $(0,\frac{2}{3})_s\times\stab_{\nabla f_j}\times D^k_{(x_1,\cdots, x_k)}$ and $\{\frac{2}{3}\}_s\times\stab_{\nabla f_j}\times D^k_{(x_1,\cdots, x_k)}$ respectively.
			
			
%

			The values of $\phi_0$ on $\Lc$ are bounded between the $\phi_0$ value of $\Delta$ and the $\phi_0$ value of $\Delta_j$. Since $V_0$ is non-vanishing along $\{\frac{1}{2}\}\times \cK_j\times\{0\}$, we can homotope $\phi_0$ through gradient-like functions for $V_0$ which are unchanged outside of a neighborhood of $\Lc$ so that we can assume $\phi_0$ takes a constant value $c$ on $\{\frac{1}{2}\}\times \cK_j\times\{0\}$ and in a neighborhood of $\cK_j\times\{0\}$ on $\Delta_j$ is given by $\phi_0(s,k,\vec{x})=c+s-\sum_i x_i^2$. 
			
			Now choose a birth family of functions $\eta_t:(0,1)\to \R$ where $\eta_0(s)=s+c-\frac{1}{2}$, $\eta_{\frac{1}{2}}$ has an embryonic point at $\frac{1}{2}$ and $\eta_1$ has a local maximum at $\frac{1}{3}$ and a local minimum at $\frac{2}{3}$. Let $\bar{\eta}_t(s,k,\vec{x})=\eta_t(s)+f_j(k)-\sum_i x_i^2$. Let $K\subset \cK_j$ be a compact subset which contains all of $\cK_j$ except a small neighborhood of its joints on other bones. Let $\rho:W\to \R$ be a bump function which is identically $1$ on a small neighborhood of $[\frac{1}{3},\frac{2}{3}]\times K\times \{0\}$ and vanishes outside a slightly larger neighborhood. Let $\delta(t)$ be a smooth function which is identically $0$ near $t=0$ and identically $1$ near $t=1$. Then define
				$$\phi_t = (1-\delta(t)\rho)\phi_0+\delta(t)\rho\bar{\eta}_t$$
			Since $V_0$ is non-vanishing on $\Delta_j\cap \nu$ and gradient-like for $\phi_0$, there exists a metric $g$ such that with respect to this metric, $V_0=\nabla_g \phi_0$ in the neighborhood of $[\frac{1}{3},\frac{2}{3}]\times \cK_j\times \{0\}$. Now let $Y_t$ be a vector field on $\Delta_j$ given in this neighborhood $\nabla_g(\phi_t|_{\Delta_j})$ and extended by $V_0$ outside. Then apply Proposition \ref{p:isotropichtpy} to this family. Letting $\nu_0 = \{0\leq s\leq \frac{2}{3}\}$, $\nu_1 = \{0\leq s\leq \frac{1}{3} \}$, and $\widetilde{\nu}\subset \nu_0$ disjoint from the neighborhood of modification, the resulting Weinstein structure has the first three properties listed in the proposition.
			
			Because $\cK$ has only arboreal singularities, as we modify the Weinstein structure on each $\Delta_j$, we can ensure that $V_t$ remains tangent to the original skeleton in a small neighborhood $\widetilde{U}$ of the bone we are modifying by the last point in Proposition \ref{p:isotropichtpy}. However, there may be a transitional region where changes in $V_t$ yield changes in the skeleton. This transitional region is a trivial Weinstein cobordism which may have non-trivial holonomy. The skeleton is changed because $Skel(W,\omega, V,\phi)\cap \partial \nu$ may flow to a different (but contact isotopic) Legendrian under this non-trivial holonomy which then could have a qualitatively different front projection onto $\Delta_j$ as it flows to the joints. To prevent these qualitative differences in the singularity structure of the skeleton, we use a simple observation we will refer to as the \emph{reverse holonomy trick}. 
			
			Using the neighborhoods $\widetilde{U}\subset U$ of $\Delta_j$ provided by Proposition \ref{p:isotropichtpy}, choose a slightly larger regular neighborhood $\hat{U}\supset U$ such that $V_t$ is non-vanishing in $\hat{U}\setminus U$. Let $\Gamma_1$ be the holonomy of the trivial Weinstein cobordism using the structure $(\omega_1,V_1,\phi_1)$ from the convex contact boundary of $U$ to the convex contact boundary $\widetilde{U}$. Let $\Gamma_0$ be the corresponding holonomy on this cobordism using the structure $(\omega_0,V_0,\phi_0)$. Let $\Gamma_2$ be the holonomy of the trivial Weinstein cobordism from the convex contact boundary of $\hat{U}$ to the convex contact boundary of $U$ (which is the same for the structures $(\omega_t,V_t,\phi_t)$ since the Weinstein homotopy is supported in $U$). (We follow the convention of \cite{CE} that the holonomy is the contactomorphism \emph{from} the positive boundary \emph{to} the negative boundary obtained by following the negative flow of $V_1$). Now homotope the Weinstein structure $(\omega_1,V_1,\phi_1)$ in $\hat{U}\setminus U$ to change its holonomy from $\Gamma_2$ to $\Gamma_1^{-1}\circ\Gamma_0\circ\Gamma_2$. Then the total holonomy from $\partial \hat{U}$ to $\partial \widetilde{U}$ is $\Gamma_1\circ \Gamma_1^{-1}\circ\Gamma_0\circ\Gamma_2=\Gamma_0\circ \Gamma_2$. Rename the modified structure $(\omega_1,V_1,\phi_1)$. Then $Skel(W,\omega,V,\phi)$ agrees with $Skel(W,\omega_1,V_1,\phi_1)$ in $\widetilde{U}$, and in particular the joints of $Skel(W,\omega_1,V_1,\phi_1)$ on $\Delta_j$ are unchanged. 
			
			The skeleton in the transitional region $\hat{U}\setminus \widetilde{U}$ may differ from the original skeleton by a Lagrangian isotopy. This isotopy initially may affect the singularity structure of the skeleton near the joints of $\Delta_j$ on other bones including $\Delta$. To avoid such changes, we use the reverse holonomy trick again now in neighborhoods other bones which $\Delta_j$ has joints on going in $\phi$-value decreasing order up through the joints on $\Delta$. 
		\end{proof}
		
		After applying this procedure to isolate the Liouville flow into small neighborhoods of the joints, the stable manifolds of components of zeros of the Liouville vector field are no longer all Lagrangian bones. The Legendrian $\cK_0$ is invariant under the flow of $V_1$ and is made up of a collection of $(n-1)$ dimensional stable manifolds of connected components of zeros. Applying Proposition \ref{p:thicken}, we can thicken each of these to $n$ dimensional Lagrangian bones where the thickening of $Z_0$ to $Z_1$ is in the Reeb direction, while preserving the other bones in the skeleton by choosing the unstable manifolds of the zeros in $Z_1$ to contain the span of $V_0$ in their tangent spaces. We call the resulting collection of bones lying in $\partial \nu_0$, the \emph{Reeb ribbon}.
		
		After this thickening, there are two sets of bones with joints given by the core $\cK_0\times\{0\}$ on the Reeb ribbon, but with opposite co-orientations. That the joints coincide is not generic. We can incorporate a Reeb flow contact isotopy into the holonomy on one side to make these two sets of joints disjoint. When the Legendrian $\cK_0$ is smooth, this suffices to make the joints generic and the singularities on the thickening of $\cK_0$ arboreal $A_2$ singularities. However, when $\cK_0$ has singularities, the joints require further contact perturbations to be generic. This is because the skeleton in a neighborhood of the Reeb ribbon is a product of \includegraphics[scale=.15]{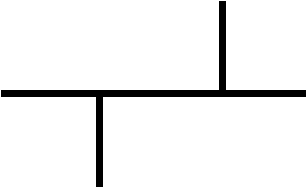} with $\cK_0$. Even if $\cK_0$ has only arboreal singularities, the product of arboreal singularities is not arboreal (see figure \ref{fig:A2prod}). 
		
		\begin{figure}
			\centering
			\includegraphics[scale=.4]{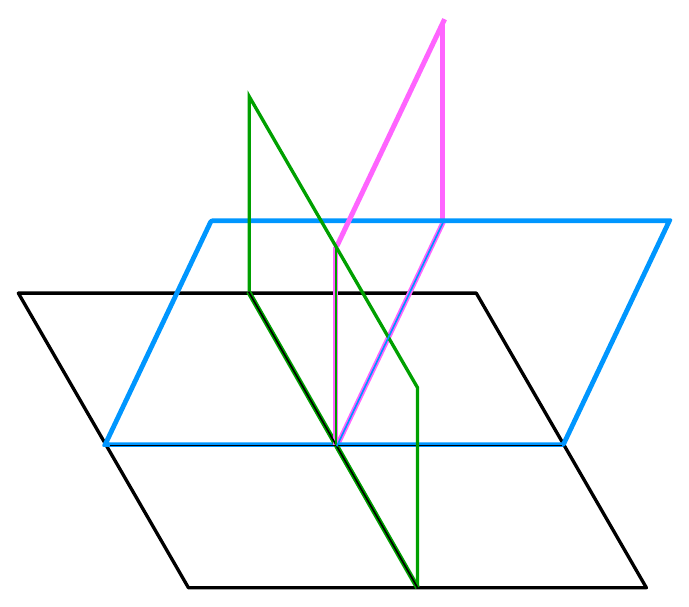} \hspace{1cm} \includegraphics[scale=.4]{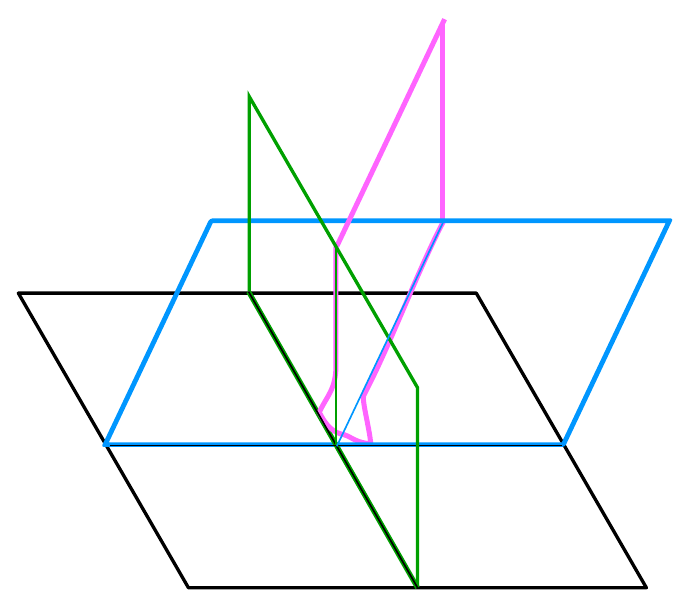}
			\caption{The product of an $A_2$ singularity with itself has four Lagrangian strata which all intersect at a point, and thus is not arboreal (left). However, a holonomy perturbation adjusts it to a collection of arboreal singularities (right). This is the simplest case of a more general resolution procedure of Nadler. Note the figure represents a skeleton embedded in $\R^4$ so interpret with care.}
			\label{fig:A2prod}
		\end{figure}
		
		The problem is that the product of two singularities has too much symmetry, and is not stable under perturbations of the holonomy. Nadler has an explicit resolution of a product of arboreal singularities into a collection of arboreal singularities. For our purposes, we know by Theorem \ref{thm:notang}, that it suffices to use holonomy perturbations to make all joints immersed. We can use the perturbations described in the proof of theorem \ref{thm:model} to make the joints immersed between bones which were not in generic position due to the product structure. Many of the joints are already immersed because at the end of the Weinstein homotopy of Proposition \ref{p:localize}, when the Reeb ribbon is an $(n-1)$ dimensional isotropic submanifold, the projection of $\cK_1$ to $\cK_0$ is an embedding because the skeleton in this region is Lagrangian isotopic to the product $\cK_1\times I$ where the $I$ direction is given by the flow of $V_0$. After thickening the Reeb ribbon, the unstable manifolds which are projected out in the foliation become smaller dimensional, so the rank of the front projection remains at least as large as it was before thickening. The non-immersed joints come from higher co-dimension strata in $\cK_0$ which can become joints between some of the bones which have been producted with \textbar \, onto some of the bones which have been producted with \textemdash. By choosing our holonomy perturbations of the product singularities to be sufficiently small, they will not create new singularities of tangency.

	\subsection{Arborealizing $\Sigma^{1,0}$ tangential singularities} \label{s:arborealize}
	
	Finally, we prove our strongest result by performing the final step in arborealization for the first class of tangential singularities. The mechanism to eliminate these singularities of tangency is similar to Nadler's real blow-up of \cite{N2}. The main difference is that we spread out these singularities at the level of the Legendrian instead of at the level of the front projection. This allows us to control the $C^1$ size of the modification which in turn allows us to avoid introducing additional singularities on the Reeb ribbon which was produced in the localization stage. An example indicating how the proof works is shown in figure \ref{fig:arborealization}.
	
	\begin{figure}
		\centering
		\includegraphics[scale=.5]{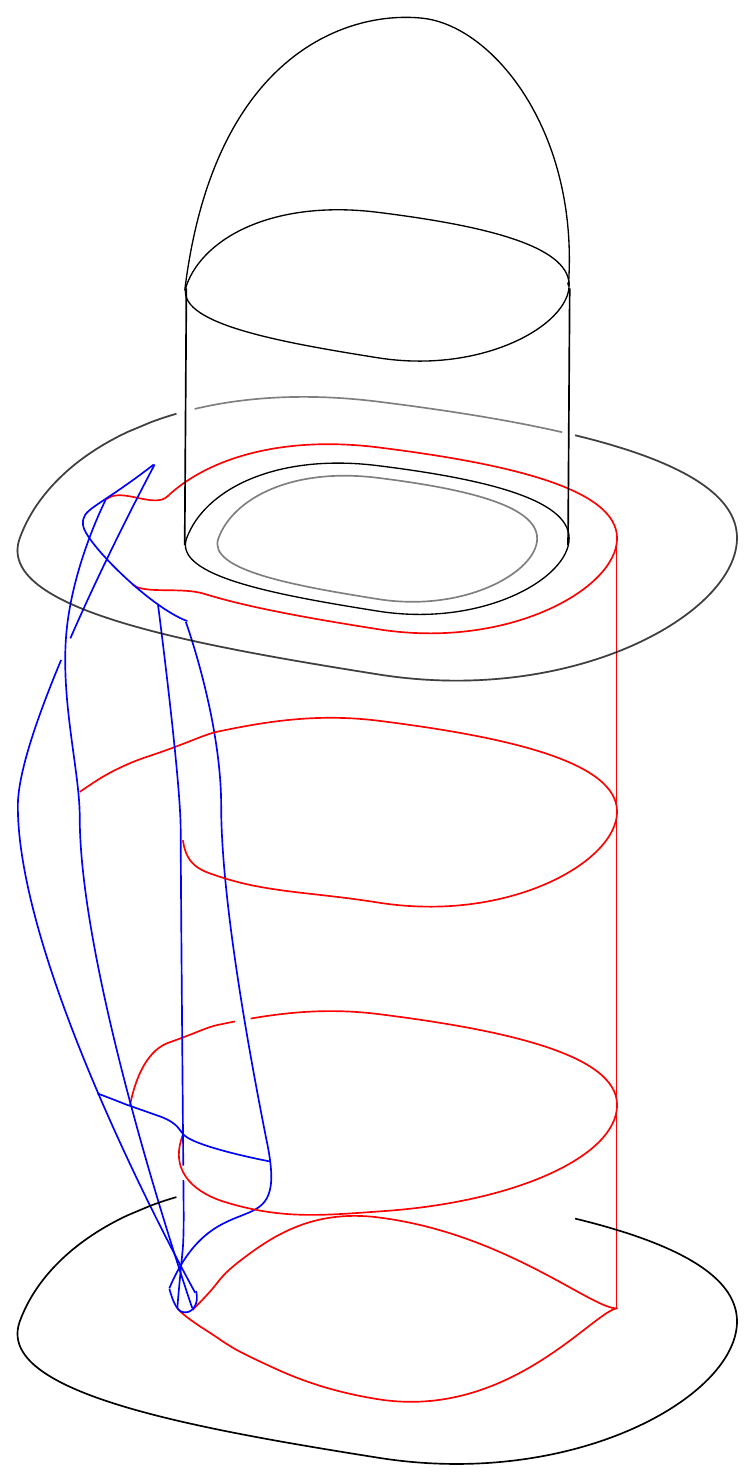}
		\caption{An example of localization, introduction of the Reeb ribbon, and arborealization of one $\Sigma^{1,0}$ cusp singularity of tangency.}
		\label{fig:arborealization}
	\end{figure}
	
	A bone $\Delta_j$ has a non-empty joint $\cH$ on $\Delta$ if and only if $\Lc:=\stab(\Delta_j)\cap \unstab(\Delta)\neq \emptyset$ where $\Lc$ is identified with the co-normal to $\cJ:=\cH\cap Z$ in $T^*Z$. We say that the joint $\cH$ has a \emph{tangential singularity} if the corresponding front projection $\Lc\to Z$ has a singularity of tangency. These singularities of tangency can be organized and characterized via the Thom-Boardman stratification as described in section \ref{s:front}. If all singularities of tangency have Thom-Boardman type $\Sigma^{1,0}$ we can explicitly eliminate these tangential singularities.
	
	\begin{theorem1}[\textbf{\ref{thm:sigma10}}]
		Suppose all of the tangential singularities of joints have type $\Sigma^{1,0}$. Then there is a homotopy of the Weinstein structure to one whose skeleton has only arboreal singularities.
	\end{theorem1}
	
	\begin{proof}
		Suppose $\Delta_j$ has a joint on $\Delta$ with $\Sigma^{1,0}$ singularities. In other words, if $Z$ is the marrow of $\Delta$ and $\cK_j=\Delta_j\cap \unstab(Z)\cap \partial \nu_1 \subset S^*Z$, the front projection $\cK_j\subset S^*Z \to Z$ has $\Sigma^{1,0}$ singularities. We will assume that we have localized as in section \ref{s:localize} and use the notation and coordinates of that section identifying $\cK_j$ with $\{\frac{1}{3}\}\times \cK_j\times\{0\}\subset \partial \nu_1\cap \unstab(Z)$, noting that this isotropic is invariant under the flow of $V_1$.  Recall that $V_0$ was a Liouville vector field outwardly transverse to $\nu_1$ in a neighborhood of $\cK$ and thus induces a contact form. Near $\cK_j$, we will use the non-vanishing vector field $V_0$ to integrate to a coordinate $s$ so that $V_0=\partial_s$ near $\cK_j$ and $\partial \nu_1 = \{s=\frac{1}{3}\}$. Also note that the skeleton with respect to $V_1$ is also invariant in the $s$ direction near $\cK_j$, and this $s$ direction corresponds to a radial direction in the co-tangent fibers $T^*Z$. We will work entirely in the $T^*Z$ slice where $\vec{x}=\vec{y}=0$ where $\vec{x}$ are the stable coordinates in $\Delta$ transverse to $Z$ and $\vec{y}$ are unstable coordinates integrating the kernel of the symplectic form on $\unstab(Z)$. The full joint on $\Delta$ is simply given by the product with $\R^k_{(x_1,\cdots, x_k)}$ of the intersection of the joint with $Z$, so it suffices to eliminate singularities of tangency in $Z$.
		
		Because of localization, there are two front projections of $\cK_j$ appearing as joints in the skeleton: one to $Z$ whose tangential singularities we are trying to resolve, and one to the Reeb ribbon which is an embedding to the core of the Reeb ribbon with no singularities of tangency. Let $\cF$ denote the Legendrian foliation corresponding to the front projection to $Z$. Let $\cG$ denote the Legendrian foliation corresponding to the front projection to the Reeb ribbon. If the singular locus of the front projection of $\cK$ with respect to $\cF$ has only type $\Sigma^{1,0}$ then there is a smooth hypersurface $\Sigma\subset \cK_j$ (possibly with boundary in $\partial \cK_j$) where the intersection between the tangent space to the skeleton and the tangent space to the leaves of $\cF$ is a 1-dimensional line field which is tangent to the skeleton, but transverse to $\Sigma$.
		
		Let $d$ denote the dimension of $\cK_j$ ($d=n-k-1$). Because $\Sigma^{1,0}$ singularities have a standard model, we have local coordinates $(q_1,\cdots, q_{d-1},u)$ on $\cK_j$ near each point in $\Sigma$ and local coordinates $(z_1,\cdots, z_{d+1})$ near its image in $Z$ such that the front projection $\pi_{\cF}|_{\cK_j}: \cK_j\to Z$ is given by $\pi_{\cF}|_{\cK_j}(q_1,\cdots, q_{d-1},u) = (q_1,\cdots, q_{d-1},u^2,u^3)$. In these coordinates $\Sigma = \{u=0\}$, $\ker(d\pi_{\cF})$ is spanned by $\partial_u$ along $\Sigma$, and $(q_1,\cdots, q_{d-1})$ give local coordinates on $\Sigma$. Let $v:\partial \nu_1 \to \R$ be the coordinate defined by $v=\pi_{z_d}\circ \pi_{\cF}$. Then near $\Sigma$, $\cK_j\subset \{v=u^2\}$ and the leaves of $\cF$ are contained in the level sets of $v$. Notice that $\partial_v$ is transverse to $T\cK + T\cF$ along $\Sigma$. Since $\partial_u\in T\cK\cap T\cF$ and $\cK$ and $\cF$ are Legendrian, $\partial_v$ must pair symplectically with $\partial_u$, thus $(u,v)$ splits off a symplectic factor normal to $\Sigma$. Since $\Sigma$ is isotropic and its symplectic normal bundle in $T^*Z$ is trivialized by the Lagrangian frame $(\partial_s, \partial_u)$ it has a neighborhood in $T^*Z$ symplectically modeled on $T^*\Sigma\times \C_{(u,v)} \times \C_{(s,r)}$. Along $\Sigma$, the leaves of $\cF$ are isotropic and transverse to $\Sigma\times \C_{(u,v)}\times \C_{(s,r)}$. Therefore we can choose co-dimension $1$ slices of these leaves along $\Sigma$ symplectically orthogonal to $\Sigma\times \C_{(u,v)}\times \C_{(s,r)}$, and use this to  identify coordinates on the neighborhood of $\Sigma$ so that these leaves coincide with the cotangent fibers in $T^*\Sigma$ in the splitting $T^*\Sigma\times \C_{(u,v)}\times \C_{(s,r)}$. Denote the corresponding coordinates on $T^*\Sigma$ dual to $(q_1,\cdots, q_{d-1})$ by $(p_1,\cdots, p_{d-1})$.
		
		Choose compact subsets $K\subset K'\subset \cK_j$ containing all but a small neighborhood of the joints of $\cK_j$ on other bones. Choose a Morse-Bott* function $\psi: \cK_j\to \R$ such that with respect to a fixed metric, $\nabla \psi$ agrees with $V_1$ outside of $K'$, $\Sigma$ agrees with a union of stable manifolds of critical points of $\psi$ inside $K$, and the same union of stable manifolds of critical points of $\psi$ is $C^1$ close to $\Sigma$ in all of $\cK_j$. If $\Sigma$ has no boundary, we can assume that there is a unique connected component of critical points of $\psi$ whose stable manifold is $\Sigma$. If $\Sigma$ has boundary, then we may need a second connected component of critical points of $\psi$ so that the union of the stable manifolds includes $\partial \Sigma$. Denote the points making up this critical set by $\Lambda$ so $\stab(\Lambda)$ agrees with $\Sigma$ in $K$. Use Proposition \ref{p:isotropichtpy} and the reverse holonomy trick as in the previous section to create a Weinstein homotopy to $(W,\omega_2,V_2,\phi_2)$ preserving the skeleton up to Lagrangian isotopy so that $V_2$ restricted to $\cK_j$ agrees with a scalar multiple of the gradient of $\psi$. 
		
		Now, the critical points of $\psi$ of less than maximal index will have stable manifolds which are isotropic but not Lagrangian, so we can thicken them. We will take particular care in thickening the stable manifolds of $\Lambda$ to eliminate the $\Sigma^{1,0}$ tangential singularities. Choose $\varepsilon>0$ sufficiently small so that the coordinate model above is defined for $|u|<3\varepsilon$.
		
		Consider the $C^1$ functions
		$$\tilde{h}_+(u) = \begin{cases}  4\varepsilon u-4\varepsilon^2 & u\leq 2\varepsilon\\ u^2 & u\geq 2\varepsilon\end{cases}$$
		$$\tilde{h}_-(u) = \begin{cases}  -4\varepsilon u-4\varepsilon^2 & u\geq -2\varepsilon\\ u^2 & u\leq -2\varepsilon\end{cases}$$
		Let $h_\pm$ be a $C^1$ close smoothing of $\tilde{h}_\pm$. Note that the graph $v=h_+(u)$ intersects $v=\varepsilon u$ at the point $u=\frac{4}{3}\varepsilon$, $v=\frac{4}{3}\varepsilon^2$ and the graph $v=h_-(u)$ intersects $v=\varepsilon u$ at the point $u=-\frac{4}{5}\varepsilon$, $v=-\frac{4}{5}\varepsilon^2$. Also observe that choosing $\varepsilon$ sufficiently small makes $h_+(u)$ $C^1$ close to $u^2$ for $u\geq \frac{4}{3}\varepsilon$ and $h_{-}(u)$ $C^1$ close to $u^2$ for $u\leq -\frac{4}{5}\varepsilon$.
		
		Thicken the zero set $\Lambda$ in a neighborhood with $|u|<3\varepsilon$ and $|v|<4\varepsilon^2$ to a thickened zero set $\overline{\Lambda}$ so that locally $\stab(\overline{\Lambda}) = \{\vec{p}=0, r=s=0, v=\varepsilon u, |u|\leq 2\varepsilon$. Arrange the transverse Lagrangian foliation such that for points $(\vec{q}_0,u=\frac{4}{3}\varepsilon)\in \overline{\Lambda}$ the unstable manifold of that point is $\mathcal{U}_{\vec{q}_0,+}:=\{v=h_+(u), \vec{q}=\vec{q}_0, r=0\}$ and for points $(\vec{q}_0, u=-\frac{4}{5}\varepsilon)\in \overline{\Lambda}$ the unstable manifold of that point is $\mathcal{U}_{\vec{q}_0,-}:=\{v=h_-(u), \vec{q}=\vec{q}_0,r=0\}$. These unstable manifolds agree near the boundary of the neighborhood with the skeleton of the Weinstein structure before thickening $\Lambda$. The thickening (Proposition \ref{p:thicken}) uses a cut-off function to splice in the thickened Weinstein structure. We can choose this cut-off function $\rho$ such that wherever the skeleton intersects $\{\rho'\neq0\}$ in $T^*Z$, $\rho$ depends only on $v$. The spliced Liouville vector field is given by $(1-\rho)V_1+\rho V_{loc}+HX_\rho$ where $H$ is a Hamiltonian function such that $V_1\setminus V_{loc}=X_{H}$. If $\rho$ depends only on $v$ in a neighborhood in $T^*Z$ of the skeleton then $X_\rho$ is a multiple of $\partial_u$ which is tangent to the leaves of $\cF$. We can apply the reverse holonomy trick in this transitional region to undo any changes to the $u$-coordinates of the skeleton as it enters the neighborhood of $\overline\Lambda$ without changing the invariance of the skeleton in the $q_i$ directions. In a neighborhood of $\overline{\Lambda}$, the skeleton agrees with the union of $\stab(\overline\Lambda)$ with $\cup \mathcal{U}_{\vec{q_0},+}\cup\mathcal{U}_{\vec{q_0},-}$. See the $(u,v)$ slice of this skeleton in figure \ref{fig:nosigma10}. Holonomy changes outside a neighborhood of $T^*Z$ effected by the cut-off will only effect lower $\phi$-valued joints and thus will be dealt with at a later stage. Let $(V_2,\phi_2)$ denote the resulting Weinstein structure.
		
		\begin{figure}
			\centering
			\includegraphics[scale=.7]{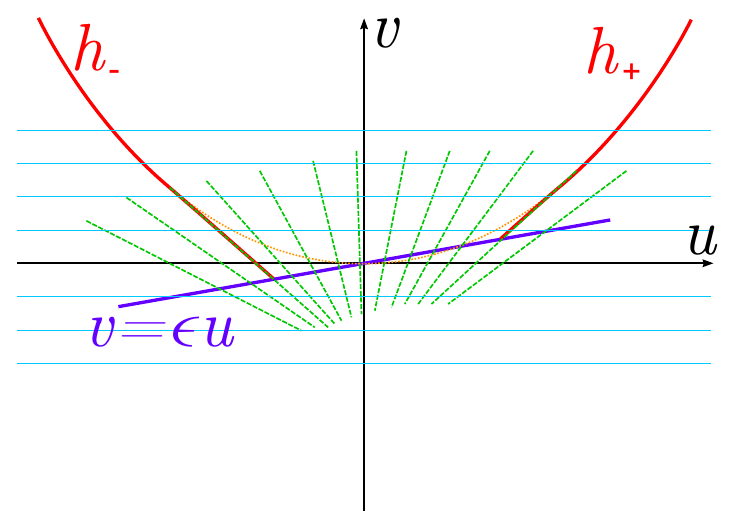} \hspace{1cm} \includegraphics[scale=.7]{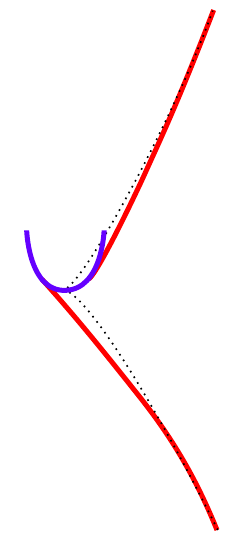}
			\caption{Eliminating $\Sigma^{1,0}$ tangencies. Left: the Lagrangian projection in the $(u,v)$ coordinates, where horizontal lines show the leaves of the foliation $\cF$. Right: the corresponding slice of the front projection, where the black dotted curve indicates the original cusp curve. The vertical coordinate here is obtained by integrating $u$ as a function of $v$ and the horizontal coordinate is $v$.}
			\label{fig:nosigma10}
		\end{figure}
		
		Now we verify the resulting front projections $\cF$ and $\cG$ have no remaining singularities of tangency. Note that outside of the neighborhood where our coordinates are defined, there are no singularities of tangency. For the bone $\stab(\overline\Lambda)$, the tangent space in $TS^*Z$ is spanned by $(\varepsilon \partial_v+\partial_u, \partial_{q_i})$, whereas the leaves of $\cF$ have tangent spaces spanned by $(\partial_u, \partial_{p_i})$, so there is no non-trivial intersection of the tangent spaces. Since $\Sigma$ agrees in $C$ with $\stab(\Lambda)$, the other pieces of the skeleton in this neighborhood agree near $\Sigma$ with the sets 
		$$\{v=h_+(u), \vec{p}=0, r=0, u\geq \frac{4}{3}\varepsilon \} \text{ and } \{v=h_-(u), \vec{p}=0, r=0, u\geq -\frac{4}{5}\varepsilon \}$$ 
		and the tangent space to these strata are spanned by $(\partial_u+h'_\pm(u)\partial_v, \partial_{q_i})$. Since $h'_\pm(u)\neq 0$, these tangent spaces intersect trivially with $T\cF$. In the transitional holonomy region, $V_2$ is a linear combination of two Liouville vector fields which preserve $\cK_j$ plus some function multiple of $\partial_u$. Since $T_p\cK_j\cap T_p\cF_p=\{0\}$ in this region which lies away from $\Sigma$ and $\partial_u\in T_p\cF_p$, the skeleton of $(W,\omega,V_2,\phi_2)$ has no singularities of tangency with respect to $\cF$.  
		
		Regarding the front projection with respect to the foliation $\cG$ to the Reeb ribbon, we know that before the thickening of $\Lambda$, $\cK_j$ had no singularities of tangency with $\cG$. We can choose $\varepsilon$ sufficiently small so that the skeleton after thickening is $C^1$ close to the skeleton before thickening, therefore no singularities of tangency with respect to $\cG$ are introduced. 
		
		Similarly, by performing localization and elimination of singularities of tangency starting from the highest $\phi$-valued joints to the lowest $\phi$-valued joints, we ensure that the joints of $\cK_j$ onto other bones have no singularities of tangency (these would have been eliminated at an earlier iteration). Therefore, because the modified skeleton remains $C^1$ close to the original, we can arrange that these joints in the horizontal direction (within $\partial \nu_1\cap \unstab(Z)$) do not acquire any new singularities of tangency.
	\end{proof}

	\bibliography{references}
	\bibliographystyle{alpha}

\end{document}